\documentclass[12pt,a4paper,reqno,oneside]{amsart}
\usepackage[margin=2.5cm]{geometry}
\usepackage[utf8]{inputenc}
\usepackage[english]{babel}

\usepackage{svg}
\usepackage{bbm}
\usepackage{lipsum}
\usepackage{mathrsfs}
\usepackage{amsfonts}
\usepackage{amsmath}
\usepackage{amssymb}
\usepackage{geometry}
\usepackage{amsthm}
\usepackage{enumerate}
\usepackage[dvipsnames]{xcolor}
\usepackage{esint}
\usepackage{tcolorbox}
\definecolor{mygray}{gray}{0.8}
\usepackage{cancel}

\newtheorem{theorem}{Theorem}[section]
\newtheorem{lemma}[theorem]{Lemma}

\newtheorem{proposition}[theorem]{Proposition}
\theoremstyle{definition}
\newtheorem{definition}[theorem]{Definition}
\newtheorem{remark}[theorem]{Remark}

\newcommand{\gen}{\operatorname{gen}}
\newcommand{\dist}{\operatorname{dist}}
\newcommand{\essinf}{\operatorname{ess\,inf}}

\usepackage{hyperref}
\hypersetup{
	  colorlinks,%
	citecolor=blue,%
	filecolor=red,%
	linkcolor=red,%
	urlcolor=blue
}

\allowdisplaybreaks

\begin{document}
	
	
\title[New characterizations of Muckenhoupt $A_p$ distance weights for $p>1$]{New characterizations of Muckenhoupt $A_p$ distance weights for $p>1$}
	
\author[Ignacio G\'{o}mez Vargas]{Ignacio G\'{o}mez Vargas}
\begin{abstract}
  We characterize the collection of sets \( E \subset \mathbb{R}^n \) for which there exists \( \theta \in \mathbb{R}\setminus\{0\} \) such that the distance weight \( w(x) = \operatorname{dist}(x, E)^\theta \) belongs to the Muckenhoupt class \( A_p \), where \( p > 1 \). These sets exhibit a certain balance between the small-scale and large-scale pores that constitute their complement—a property we show to be more general than the so-called weak porosity condition, which in turn, and according to recent results, characterizes the sets with associated distance weights in the \( A_1 \) case. Furthermore, we verify the agreement between this new characterization and the properties of known examples of distance weights, that are either \(A_p\) weights or merely doubling weights, by means of a probabilistic approach that may be of interest by itself.
  
  
\end{abstract}

\maketitle

\section{Introduction}
\label{sec1}

The theory of Muckenhoupt weights has been a central part of harmonic analysis for decades due to their relation to the boundedness of important operators such as the Hardy–Littlewood maximal operator and Calderón–Zygmund operators. These properties make them suitable for obtaining various Sobolev-type inequalities and embeddings, which are particularly useful for proving regularity properties of solutions to PDEs on a given domain $\Omega\subset \mathbb{R}^n$. In this regard, and especially when the boundary $\partial\Omega$ is not smooth enough to allow for the existence of well-behaved solutions in standard Sobolev spaces, the use of weights of the form $x \mapsto \text{dist}(x,E)^\theta$, where $E \subset \partial\Omega$, has been considered by numerous authors, since these weights are sometimes able to diminish the effects of boundary roughness on the behavior of the solution and its derivatives (see \cite{LOPEZ,AIMAR} and references therein). We refer to functions of this sort as \textit{distance weights}.

Due to their aforementioned applications, the problem of determining when a set admits distance weights belonging to some $A_p$ class has become of increasing interest lately. Let us observe that classical examples of $A_p$ weights, often found in textbooks, are given by $w(x) = |x|^\alpha$ with $-n < \alpha < n(p - 1)$, which can be considered as distance weights associated with the set $E = \{0\}$. Furthermore, several articles have contributed to finding sufficient conditions on the set $E$ to ensure the desired properties of its distance weights, both in $\mathbb{R}^n$ and in more general spaces. These conditions usually rely on measure-theoretical characterizations of such sets, such as basic notions of porosity, dimensional bounds, Ahlfors regularity and more \cite{AIKAWA,AIMAR,DYDA,VAHAKANGAS}.

Focusing on the $p = 1$ case—meaning when $\text{dist}(\cdot, E)^\theta \in A_1$—we can affirm that there has been major progress in the intended characterization. The first step was taken by A. V. Vasin in his study of BLO functions on the unit circle, giving rise to the concept of \textit{weak porosity} \cite{VASIN}. Years later, Anderson et al. extended this notion to $\mathbb{R}^n$ and proved that a set $E$ is weakly porous if and only if $\text{dist}(\cdot, E)^{-\alpha} \in A_1$ for some $\alpha > 0$ \cite{ANDERSON}—it is worth mentioning here that only negative powers can give rise to non-trivial distance weights when $p = 1$. Since then, weakly porous sets have been defined in more abstract spaces in subsequent works \cite{NOSOTROS,MUDARRA}. See also \cite{NOS_LATERAL} for a one-sided version of this notion. Nevertheless, in this paper, we will limit our discussion to $\mathbb{R}^n$. Roughly speaking, a weakly porous set $E\subset\mathbb{R}^n$ is a set such that all cubes $Q$ in $\mathbb{R}^n$
contain a finite sequence of pairwise disjoint dyadic subcubes $Q_1,\hdots,Q_N$ which do not intersect $E$ and whose measures sum at least a fixed proportion of the measure of $Q$ (see Section \ref{sec6} for further details). By recursively applying the former condition on a cube $Q$ and some of its smaller subcubes, it is possible to prove that the fraction of $Q$ occupied by a collection of dyadic cubes decreases exponentially as their lengths get smaller, showing that the weak porosity condition has very strong implications for the global distribution of pores within any given cube.

In this work, we consider the task of characterizing the sets $E \subset \mathbb{R}^n$ that admit distance weights in the $A_p$ class for $1 < p < \infty$, which requires proposing a more general condition than weak porosity on such sets due to the fact that $A_1 \subsetneq A_p$. As a result, we obtain two new characterizations based on the distribution of dyadic pores in $Q \setminus\overline{E}$, where $Q$ is any cube in $\mathbb{R}^n$. Among these, the second characterization is not only somewhat simpler, but also more reminiscent of the weak porosity condition, and can be stated as follows. For a fixed cube $Q_0\subset\mathbb{R}^n$, consider the collection $$\mathcal{D}_E(Q_0)=\{Q\in\mathcal{D}(Q_0):Q\cap \overline{E}=\emptyset\land\pi Q\cap \overline{E}\neq\emptyset\},$$
where $\mathcal{D}(Q_0)$ is the dyadic family of subcubes of $Q_0$ and $\pi Q$ denotes the dyadic parent of a given $Q\in\mathcal{D}(Q_0)$ (we refer to Section \ref{sec2} for further details). The condition to be established is then the existence of constants $0<s<1$ and $C_0>0$ such that
\begin{equation}
    \label{eq0}
    0<\frac{\sup\{L\geq0:\sum_{Q\in\mathcal{D}_E(Q_0)\land l(Q)\geq L}|Q|\geq s|Q_0|\}}{\inf\{L\geq0:\sum_{Q\in\mathcal{D}_E(Q_0)\land l(Q)\leq L}|Q|\geq s|Q_0|\}}\leq C_0
\end{equation}
for every cube $Q_0\subset\mathbb{R}^n$ intersecting $\overline{E}$ and where $l(Q)$ has been used to denote the side length of $Q$. The main result of this paper is the following.

\begin{theorem}
    \label{teo0}
    Let $E\subset\mathbb{R}^n$ be a non-empty set. Then, the next statements are equivalent.
    \begin{itemize}
        \item[(I)] There exist $C_0>0$ and $0<s<(1+2^n)^{-1}$  such that \eqref{eq0} holds for every cube $Q_0\subset\mathbb{R}^n$ intersecting $\overline{E}$;
        \item[(II)] For any $p\in(1,\infty)$, there exists $\theta\neq0$ such that $\dist(\cdot,E)^\theta\in A_p$.
    \end{itemize}
\end{theorem}

When examining the condition stated in Theorem~\ref{teo0}~(i), we observe that it does not depend on the so-called maximal dyadic pore---closely related to the quantity \( \essinf_{Q_0} \dist(\cdot, E)^{-1} \), which plays a central role in the \( A_1 \) case---in contrast with the weak porosity condition to be discussed in Section~\ref{sec6}, which strongly depends on this object. Instead, there appears to be a balance between the side lengths associated with the \( s \)-fractions of both the smallest and largest dyadic pores within \( Q_0 \setminus \overline{E} \), which aligns more closely with the nature of \( A_p \) weights. We will show that, between the two, weak porosity is the more restrictive property. In particular, every weakly porous set \( E \subset \mathbb{R}^n \) satisfies condition~\eqref{eq0} for suitable values of \( s \) and \( C_0 \). Furthermore, we aim to demonstrate throughout this work that the proposed characterization has strong implications for the global distribution of pores in \( Q_0 \setminus \overline{E} \), in analogy with the phenomenon previously discussed in the \( p = 1 \) case.

The structure of the paper is as follows. In Section~\ref{sec2}, we establish some basic facts about distance weights. Section~\ref{sec3} introduces fundamental concepts such as the \( t \)-fractions of the smallest and largest pores, which will be essential in Section~\ref{sec4} for obtaining a first characterization of the relevant sets in the case \( p > 1 \). This initial characterization relies on a relation between the $t$-fractions of the smallest and largest pores for all $0<t<1$ (see Theorems \ref{teo3} and \ref{teo4}), in contrast with the second, simpler characterization proposed in Section~\ref{sec5} and which depends on a single (small enough) value $s$ (as in Theorem \ref{teo0}). In Section~\ref{sec6}, we prove that weakly porous sets satisfy this condition. Finally, a probabilistic approach will be presented to analyze examples of distance weights introduced in previous papers.

It is worthy to mention that shortly after a first draft of this paper was made public, an independent work by M. Pasquariello and I. Uriarte-Tuero appeared \cite{PASQUARIELLO}, in which a very similar version of Theorem~\ref{teo0} was proved. Their approach differs substantially from the one discussed here: they introduce a new characterization of BMO functions obtained through a sparse domination formula by medians, which they subsequently use to derive their characterization for distance weights. In addition, they investigate the range of exponents $\theta$ for which the weight $w(x)=\dist(x,E)^\theta$ belongs to $A_p$, provide new concrete examples of such weights, and examine the validity of Poincaré and Hardy–Sobolev weighted inequalities.

\section{Preliminary results on distance weights}
\label{sec2}

The setting of this work will be the Euclidean space $\mathbb{R}^n$ and,  as usual, we denote by $|A|$ the Lebesgue measure of a measurable set $A\subset\mathbb{R}^n$. Due to the particular notions of porosity we will have to deal with, it will be convenient to make use of cubes and systems of dyadic subcubes within them. A given set $Q\subset\mathbb{R}^n$ is said to be a cube if it can be written as a product of the form
$$Q=[a_1,b_1)\times\hdots\times[a_n,b_n),$$
where $|b_i-a_i|$ is constant for every index $1\leq i\leq n$. We refer to this quantity as the side length of $Q$ and denote it by $l(Q)$. In particular, the cubes considered here are half-open and parallel to the coordinate axes.

For a given cube $Q_0\subset\mathbb{R}^n$, we refer to its dyadic system as the collection
$$\mathcal{D}(Q_0)=\bigcup_{j\in\mathbb{N}_0} \mathcal{D}^j(Q_0),$$
where each $\mathcal{D}^j(Q_0)$ is uniquely defined as the family composed by $2^{jn}$ pairwise disjoint cubes $Q$, with $l(Q)=2^{-j}l(Q_0)$ and such that $Q_0=\bigcup_{Q\in\mathcal{D}^j(Q_0)}Q$, for each $j\in\mathbb{N}_0$. The cubes $Q\in\mathcal{D}(Q_0)$ are referred to as dyadic (sub)cubes of $Q_0$, and if $Q\in\mathcal{D}^j(Q_0)$, we say that $Q$ is a dyadic cube of generation $j$ and write $\gen(Q)=j$. Clearly, the generation of any $Q\in\mathcal{D}(Q_0)$ is well-defined and unique. Other properties of dyadic cubes worthy to mention are:
\begin{itemize}
    \item[$\bullet$] If $Q\in\mathcal{D}^j(Q_0)$, where $j>0$, then there exists a unique cube $\pi Q$ of generation $j+1$ such that $Q\subset\pi Q$. We say that $Q$ is a dyadic child of $\pi Q$, meanwhile $\pi Q$ is the dyadic parent of $Q$.
    \item[$\bullet$] Each $Q\in\mathcal{D}(Q_0)$ has exactly $2^n$ dyadic children.
    \item[$\bullet$] If $Q,Q'\in\mathcal{D}(Q_0)$ and $Q\cap Q'\neq\emptyset$, then either $Q\subset Q'$ or $Q'\subset Q$.
\end{itemize}

On another note, we say a function $w:\mathbb{R}^n\to\mathbb{R}$ is a weight in $\mathbb{R}^n$ if $w>0$ a.e$.$ and if it is locally integrable. For $1<p<\infty$, the Muckenhoupt class $A_p$ is the set of all weights $w$ defined in $\mathbb{R}^n$ for which there exists some constant $C>0$ such that
\begin{equation}
    \label{eq1}
    \Big(\fint_Qw(x)dx\Big)\Big(\fint_Qw(x)^{-\frac{1}{p-1}}dx\Big)^{p-1}\leq C
\end{equation}
for every cube $Q$ in $\mathbb{R}^n$. Here, $\fint_Aw(x)dx=|A|^{-1}\int_Aw(x)dx$ stands for the mean value of $w$ on $A$, where $A\subset\mathbb{R}^n$ is a measurable set with $0<|A|<\infty$. For $p=1$, the $A_1$ class is made up of all weights $w$ for which there exists $C>0$ such that
\begin{equation}
    \label{eq2}
    \fint_Qw(x)dx\leq C\essinf_{x\in Q}w(x)
\end{equation}
for every cube $Q$ in $\mathbb{R}^n$. For a given weight $w$, the smallest possible constant for which either \eqref{eq1} or \eqref{eq2} holds is known as the $A_p$ or $A_1$ constant and denoted by $[w]_{A_p}$ or $[w]_{A_1}$, as appropriate. 

For any non-empty set $E\subset\mathbb{R}^n$, we denote by $\text{dist}(x,E)=\inf_{e\in E}|x-e|$, this is, $x\mapsto\dist(x,E)$ is the distance function to $E$. As discussed in the introduction, we are interested in determining when a set $E$ admits (non-constant) distance weights $\text{dist}(\cdot,E)^\theta\in A_p$ for some $\theta\in\mathbb{R}$ and $1\leq p<\infty$, in such a case, we colloquially refer to $E$ as an $A_p$ \textit{set}. Since $\dist(\cdot,E)=\dist(\cdot,\overline{E})$ (where $\overline{E}$ stands for the closure of $E$), clearly $E$ is an $A_p$ set if and only if $\overline{E}$ is an $A_p$ set. With this in mind, we will generally assume that $E$ is closed in order to avoid cumbersome notation. 

It turns out dyadic cubes are useful to discern when these kinds of weights happen to verify Muckenhoupt's conditions. Recall from the introduction that
$$\mathcal{D}_E(Q_0)=\{Q\in\mathcal{D}(Q_0):Q\cap \overline{E}=\emptyset\land\pi Q\cap \overline{E}\neq\emptyset\},$$
resulting in a non-empty collection whenever $Q_0$ intersects $\overline{E}$. In this case, such collection allows to cover $Q_0\setminus\overline{E}$ by a family of dyadic cubes which happen to be close to the set $E$, as shown in the following result.

\begin{lemma}
    \label{lemma1}
    If $E$ is a closed non-empty set with $|E|=0$ and $Q_0$ is a cube such that $Q_0\cap E\neq\emptyset$, then $\mathcal{D}_E(Q_0)$ is an infinite family of pairwise disjoint dyadic cubes such that
    $$\bigcup_{Q\in\mathcal{D}_E(Q_0)}Q=Q_0\setminus E.$$
\end{lemma}

\begin{proof}
    Since $|E|=0$, $Q_0\setminus E$ is not empty. Thus, for each $x$ belonging to this set we can find, using that $E$ is closed, some cube $Q\in\mathcal{D}(Q_0)$ containing $x$ and with $Q\cap E=\emptyset$. Let $Q_x$ be the largest cube with this property. Then, by maximality, $Q_x\in\mathcal{D}_E(Q_0)$. This shows that $\mathcal{D}_E(Q_0)\neq\emptyset$ and 
    $$Q_0\setminus E\subset\bigcup_{Q\in\mathcal{D}_E(Q_0)}Q\subset Q_0\setminus E.$$
    On other hand, if $Q_1,Q_2\in\mathcal{D}_E(Q_0)$ are such that $Q_1\cap Q_2\neq\emptyset$, we claim that $Q_1=Q_2$. Notice that either $Q_1\subseteq Q_2$ or $Q_1\supseteq Q_2$ because of the properties of dyadic cubes. Let us assume, without loss of generality, that $Q_1\subsetneq Q_2$. It follows that $\pi Q_1\subseteq Q_2$, but then $\emptyset\neq\pi Q_1\cap E\subset Q_2\cap E$, which contradicts the fact that $Q_2\in\mathcal{D}_E(Q_0)$ and so the claim holds. 
    
    It remains to show that $\mathcal{D}_E(Q_0)$ is infinite. Assume, on the contrary, that $\mathcal{D}_E(Q_0)=\{Q_i\}_{i=1}^N$ and let $m=\max_{1\leq i\leq N}\gen(Q_i)$. Next, pick $y\in Q_0\cap E$ and let $Q^*\in\mathcal{D}(Q_0)$ be the cube such that $y\in Q^*$ and $\gen(Q^*)=m$. Let us consider the intersection $Q^*\cap(\bigcup_{1\leq i\leq N}Q_i)$. If this intersection were empty, then
    \begin{align*}
        Q^*\subset Q_0\setminus\bigcup_{1\leq i\leq N}Q_i=Q_0\setminus\bigcup_{Q\in\mathcal{D}_E(Q_0)}Q=E,
    \end{align*}
    by the first part of the proof. But this is a contradiction since $|E|=0<|Q^*|$. So we must have $Q^*\cap Q_i\neq\emptyset$ for some $1\leq i\leq N$. Since $\gen(Q^*)\geq\gen(Q_i)$, this implies that $y\in Q^*\subset Q_i$ which, once again, is a contradiction since $Q_i\in\mathcal{D}_E(Q_0)$ and $y\in E$. Consequently, $\mathcal{D}_E(Q_0)$ contains infinitely many dyadic cubes.
\end{proof}

\begin{lemma}
    \label{lemma2}
    Fix $\theta>-1$ and let $E\subset\mathbb{R}^n$ be a closed non-empty set with $|E|=0$. Then, there exist constants $0<C_1\leq C_2<\infty$ depending only on $n$ and $\theta$ such that \begin{equation}
        \label{eq3}
        C_1\sum_{Q\in\mathcal{D}_E(Q_0)}l(Q)^{\theta}\frac{|Q|}{|Q_0|}\leq\fint_{Q_0}\dist(x,E)^{\theta}dx\leq C_2\sum_{Q\in\mathcal{D}_E(Q_0)}l(Q)^{\theta}\frac{|Q|}{|Q_0|},
    \end{equation}
    for every cube $Q_0$ for which $Q_0\cap E\neq\emptyset$. Furthermore, a constant $C_1>0$ so that the first inequality above holds can be found for every $\theta\in\mathbb{R}$.
\end{lemma}

\begin{proof}
    Fix a cube $Q_0$ satisfying $Q_0\cap E\neq\emptyset$. If $\theta=0$, then $\eqref{eq1}$ holds with $C_1=C_2=1$. Suppose then that $-1<\theta<0$. Let $Q\in\mathcal{D}_E(Q_0)$. Since $Q\cap E=\emptyset$, we have that
    $$\int_Q\dist(x,E)^{\theta}dx\leq\int_Q\dist(x,\partial Q)^{\theta}dx= C_2(n,\theta)|Q|^{1+\frac{\theta}{n}}=C_2(n,\theta)l(Q)^\theta|Q|,$$
    where the hypothesis $\theta>-1$ is needed to ensure the second integral is finite. On the other hand, if $y\in Q$, because $\pi Q\cap E\neq\emptyset$, it follows that $\dist(y,E)\leq\text{diam}(\pi Q)=C(n)l(Q)$, so $\dist(y,E)^{\theta}\geq C_1(n,\theta)|Q|^{\frac{\theta}{n}}$. This shows that
    $$\int_Q\dist(x,E)^{\theta}dx\geq C_1(n,\theta)|Q|^{\frac{\theta}{n}}\int_Qdx=C_1(n,\theta)|Q|^{1+\frac{\theta}{n}}=C_1(n,\theta)l(Q)^\theta|Q|,$$
    which holds even for $\theta\leq-1$. Recalling that $|E|=0$, we can combine the previous estimates with the fact that 
    $$\fint_{Q_0}\dist(x,E)^{\theta}dx=\fint_{Q_0\setminus E}\dist(x,E)^{\theta}dx=\sum_{Q\in\mathcal{D}_E(Q_0)}|Q_0|^{-1}\int_Q\dist(x,E)^{\theta}dx$$
    to easily arrive at \eqref{eq1}. The case $\theta>0$ follows analogously.
\end{proof}

The next two results state some basic facts about distance weights, serving as a starting point to begin tackling the considered problem.

\begin{proposition}
    \label{prop1}
    Let $E$ be a non-empty subset of $\mathbb{R}^n$. Also, let $1\leq p<\infty$ and $1<q<\infty$. If there exists $\theta\in\mathbb{R}$ such that $\dist(\cdot,E)^\theta\in A_p$, then there exists another power $\widehat{\theta}$ with the same sign as $\theta$ such that $\dist(\cdot,E)^{\widehat{\theta}}\in A_q$.
\end{proposition}

\begin{proof}
    If $u\in A_p$ and $1\leq p<q<\infty$, it is known (see \cite[Proposition 2.5 (a)]{DUOANDIKOETXEA}) that $u\in A_q$ with $[u]_{A_q}\leq[u]_{A_p}$. Meanwhile, in the case that $1<q<p<\infty$, \cite[Proposition 2.5 (b)]{DUOANDIKOETXEA}) tell us that $u^{\frac{q-1}{p-1}}\in A_q$ with $[u^{\frac{q-1}{p-1}}]_{A_q}\leq[u]_{A_p}^{\frac{q-1}{p-1}}$. The result then follows by taking $u=\dist(\cdot,E)^\theta$ and either $\widehat{\theta}=\theta$ or $\widehat{\theta}=\theta(q-1)(p-1)^{-1}$, depending on the case.
\end{proof}

\begin{proposition}
    \label{prop2}
    Let $E$ be a non-empty subset of $\mathbb{R}^n$ and assume that  $w(x)=\dist(x,E)^\theta$ belongs to $A_p$ for some $\theta\in\mathbb{R}$ and $1\leq p<\infty$. The following statements hold.
    \begin{itemize}
        \item[(i)] If $p=1$, then $w$ can only be a non-constant weight in the case $\theta<0$. 
        \item[(ii)] If $p>1$, then $w$ can be non-constant as long as $\theta\neq0$. Furthermore, if $\theta\neq0$ and we let $1<q<\infty$, there exist $\alpha,\beta>0$ such that $\dist(\cdot,E)^{-\alpha}$ and $\dist(\cdot,E)^\beta$ both belong to $A_q$.
    \end{itemize}
\end{proposition}

\begin{proof}
    If \eqref{eq2} holds for every cube $Q\subset\mathbb{R}^n$ and with $w=\dist(\cdot,E)^\theta$ and $\theta>0$, it is easy to see that $w=0$ a.e. and $\overline{E}=\mathbb{R}^n$. Also, if $\theta=0$, clearly $w=1\in A_p$ for every $1\leq p<\infty$. The weight $\dist(x,\{0\})^\alpha=|x|^\alpha$ belongs to $A_p$ when $-n<\alpha<n(p-1)$, so it serves as a non-constant example of a distance weight both when $p=1$ with $\theta<0$ and in the case $p>1$ with $\theta\neq0$. This proves (i) as well as the first assertion in (ii).

    In order to prove the second assertion in (ii), let us fix $1<p,q<\infty$ and $\theta\neq0$ until the end of the proof. By Proposition \ref{prop1}, there exists $\widehat{\theta}$ with same sign as $\theta$ such that $\dist(\cdot,E)^{\widehat{\theta}}\in A_q$. On the other hand, by the classical duality argument given by \cite[Theorem 2.1 (i)]{DUOANDIKOETXEA}), we have that $\dist(\cdot,E)^{\theta_0}\in A_{p'}$, where $\theta_0=\theta(1-p')$ and $p'$ is such that $\frac{1}{p}+\frac{1}{p'}=1$. Thus, applying once again Proposition \ref{prop1} with $p$ and $\theta$ replaced by $p'$ and $\theta_0$, we can assure the existence of some $\widehat{\theta}_0$ with opposite sign to $\theta$ such that $\dist(\cdot,E)^{\widehat{\theta}_0}\in A_q$. The proof concludes by taking either $\alpha=-\widehat{\theta}$ and $\beta=\widehat{\theta}_0$ if $\theta<0$ or $\beta=\widehat{\theta}$ and $\alpha=-\widehat{\theta}_0$ if $\theta>0$.
\end{proof}

By Proposition \ref{prop2}, not only we can omit the power $\theta=0$ in the search for sets admitting non-constant distance weights $\dist(\cdot,E)^\theta$, but we can also restrict the search for, say, negative values of $\theta$. Indeed, we will generally look for closed sets $E$ satisfying $\dist(\cdot,E)^{-\alpha}\in A_p$ for some $\alpha>0$, in line with previous works done for the $A_1$ case. 

\section{$t$-fractions of smallest and largest pores}
\label{sec3}

In the analysis of \(A_1\) sets, the maximal pore size plays a crucial role in their characterization, owing to its relationship with the essential infimum of the function \(\operatorname{dist}(\cdot,E)^{-1}\) within a given cube. We proceed to introduce this concept in a more precise manner.

\begin{definition}
    \label{def1.5}
    Let $E\subset\mathbb{R}^n$ be a non-empty set a let $Q_0$ be a cube. We denote by $\mathcal{M}_E(Q_0)$ a largest element of $\mathcal{D}_E(Q_0)$, that is, a dyadic subcube of $Q_0$ satisfying $l(\mathcal{M}_E(Q_0))\geq l(Q)$ for every $Q\in\mathcal{D}_E(Q_0)$.
\end{definition}

Notice that such a cube may not be unique; however, its side length and measure are. Also, we will generally drop the subscript $E$ and simply write $\mathcal{M}(Q_0)=\mathcal{M}_E(Q_0)$ when the context is clear. The weak porosity condition, to be discussed in greater detail in Section \ref{sec6}, provides quantitative information about statistics related to pore distribution such as
\begin{equation*}
    \sum_{\substack{Q\in\mathcal{D}_E(Q_0)\\l(Q)\geq L}}\frac{|Q|}{|Q_0|}
\end{equation*}
when $E$ is weakly porous and $L$ is comparable in some degree to $l(\mathcal{M}(Q_0))$ (cf. \eqref{eq8}). In fact, it is also possible to make similar estimates by replacing the inequality $l(Q)\geq L$ with $l(Q)\leq L$, as we will later verify in Section \ref{sec6}. Since the essential infimum of the involved weights no longer plays such an important role in the $A_p$ condition \eqref{eq1}, it is expected that the same statistics, when evaluated with values of $L$ comparable to $\mathcal{M}(Q)$, will no longer be as relevant. Nonetheless, it may still be useful to consider such statistics for arbitrary values of $L$. The upcoming definition is intended to facilitate the study of these quantities.

\begin{definition}
    \label{def2}
    Let $E\subset\mathbb{R}^n$ be a non-empty set. Given a cube $Q_0$ such that $Q_0\cap\overline{E}\neq\emptyset$ and fixed some $t\in(0,1)$, we define the sets 
    $$\mathscr{L}(t,Q_0,E)=\Bigg\{L\geq0:\sum_{\substack{Q\in\mathcal{D}_E(Q_0)\\l(Q)\geq L}}\frac{|Q|}{|Q_0|}\geq t\Bigg\}\hspace{0.25cm}$$
    and
    $$\mathscr{S}(t,Q_0,E)=\Bigg\{L\geq0:\sum_{\substack{Q\in\mathcal{D}_E(Q_0)\\l(Q)\leq L}}\frac{|Q|}{|Q_0|}\geq t\Bigg\}.$$
\end{definition}

\begin{remark}
    \label{remark0}
    The sets $\mathscr{L}(t,Q_0,E)$ and $\mathscr{S}(t,Q_0,E)$ are defined only for cubes $Q_0$ that intersect $\overline{E}$, as these are the relevant cases where proving \eqref{eq1} for a weight of the form $\dist(\cdot,E)^\theta$ poses a genuine challenge. Indeed, finding a constant $C > 0$ such that this inequality holds for a cube $Q_0$ that does not intersect $\overline{E}$ is rather trivial (cf. Cases I and II in the proof of Theorem \ref{teo4}).
\end{remark}

\begin{proposition}
    \label{prop3}
    Let $E$ be a closed non-empty set with $|E|=0$. If $Q_0\subset\mathbb{R}^n$ is a cube intersecting $E$, then the following statements hold.
    \begin{itemize}
        \item[(i)] For every $t\in(0,1)$, $\mathscr{L}(t,Q_0,E)=[0,L_t]$ where $0<L_t\leq l(\mathcal{M}(Q_0))$. In particular, $\mathscr{L}(t,Q_0,E)$ is a closed interval containing zero.
        \item[(ii)] For every $t\in(0,1)$, $\mathscr{S}(t,Q_0,E)=[L_t',\infty)$ where $0<L'_t\leq l(\mathcal{M}(Q_0))$. In particular, $\mathscr{S}(t,Q_0,E)$ is a half-closed interval which extends into infinity.
        \item[(iii)] If $0<s\leq t<1$, then $\mathscr{L}(t,Q_0,E)\subset\mathscr{L}(s,Q_0,E)$ and $\mathscr{S}(t,Q_0,E)\subset\mathscr{S}(s,Q_0,E)$.
    \end{itemize}
\end{proposition}

\begin{proof}
    In order to prove (i), start by noticing that, since $$\sum_{l(Q)\geq0}\frac{|Q|}{|Q_0|}=1,$$
    $0\in\mathscr{L}(t,Q_0,E)$ for every $t\in(0,1)$. Furthermore, if $L\in\mathscr{L}(t,Q_0,E)$ and $0\leq L^*\leq L$, then $$\sum_{l(Q)\geq L^*}\frac{|Q|}{|Q_0|}\geq\sum_{l(Q)\geq L}\frac{|Q|}{|Q_0|}\geq t,$$
    meaning $L^*\in\mathscr{L}(t,Q_0,E)$, which proves that $\mathscr{L}(t,Q_0,E)$ is an interval. Now, let $L_t:=\sup\mathscr{L}(t,Q_0,E)$. Let us verify that $L_t>0$. For that purpose, consider a sequence $\{l_n\}_{n\in\mathbb{N}}$ of positive numbers such that $l_n\searrow0$ as $n\to\infty$. Notice that, for $\mathcal{Q}_n:=\{Q\in\mathcal{D}_E(Q_0):l(Q)\geq l_n\}$, we have that 
    $$\bigcup_{n\in\mathbb{N}}\mathcal{Q}_n=\{Q\in\mathcal{D}_E(Q_0):l(Q)>0\}=\mathcal{D}_E(Q_0).$$
    Therefore, recalling that $|E|=0$, we can apply Lemma \ref{lemma1} to get
    \begin{align*}
        |Q_0|=|Q_0\setminus E|=\Big|\bigcup_{Q\in\bigcup_{n\in\mathbb{N}}\mathcal{Q}_n}Q\Big|=\lim_{n\to\infty}\Big|\bigcup_{Q\in\mathcal{Q}_n}Q\Big|=\lim_{n\to\infty}\Big|\bigcup_{l(Q)\geq l_n}Q\Big|.
    \end{align*}
    This means that, for any fixed $t\in(0,1)$, we can pick $n$ sufficiently big as to make $$\sum_{l(Q)\geq l_n}\frac{|Q|}{|Q_0|}=|Q_0|^{-1}\Big|\bigcup_{l(Q)\geq l_n}Q\Big|\geq t,$$
    thus $L_t\geq l_n>0$. To see that $\mathscr{L}(t,Q_0,E)$ is closed, it suffices to check that $L_t\in\mathscr{L}(t,Q_0,E)$. To this end, observe that $$\mathscr{L}(t,Q_0,E)\cap\{2^{-k}l(Q_0)\}_{k\in\mathbb{N}_0}=\{2^{-k}l(Q_0)\}_{k\geq k_0},$$
    for some nonnegative integer $k_0$. Actually, notice that $k_0>0$, since the fact that $Q_0\cap E\neq\emptyset$ implies $l(Q)\leq\frac{1}{2}l(Q_0)$ for every $Q\in\mathcal{D}_E(Q_0)$. We claim that $2^{-k_0}l(Q_0)=L_t$. Indeed, suppose that $2^{-k_0}l(Q_0)<L\in\mathscr{L}(t,Q_0,E)$, for some $L>0$. Then, $L<2^{-k_0+1}l(Q_0)$, otherwise that would imply that $2^{-k_0+1}l(Q_0)\in\mathscr{L}(t,Q_0,E)$, contradicting the hypothesis on $k_0$. But since there are no dyadic cubes with side lengths strictly in between $2^{-k_0}l(Q_0)$ and $2^{-k_0+1}l(Q_0)$,
    $$\sum_{l(Q)\geq 2^{-k_0+1}l(Q_0)}\frac{|Q|}{|Q_0|}=\sum_{l(Q)\geq L}\frac{|Q|}{|Q_0|}\geq t,$$
    implies that $2^{-k_0+1}l(Q_0)\in\mathscr{L}(t,Q_0,E)$, which, once again, is a contradiction. Thus, $2^{-k_0}l(Q_0)=\max\mathscr{L}(t,Q_0,E)=L_t$ and $\mathscr{L}(t,Q_0,E)$ is closed. Finally, notice that $L_t\leq l(\mathcal{M}(Q_0))+\varepsilon$ for all $\varepsilon>0$ since $$\sum_{l(Q)\geq l(\mathcal{M}(Q_0))+\varepsilon}\frac{|Q|}{|Q_0|}=0,$$
    so $L_t\leq l(\mathcal{M}(Q_0))$.

    The proof of (ii) can be obtained through similar arguments. Since
    $$\sum_{l(Q)\leq l(\mathcal{M}(Q_0))}\frac{|Q|}{|Q_0|}=1,$$
    $l(\mathcal{M}(Q_0))\in\mathscr{S}(t,Q_0,E)$ for every $t\in(0,1)$. Furthermore, if $L\in\mathscr{S}(t,Q_0,E)$ and $L^*\geq L$, then
    $$\sum_{l(Q)\leq L^*}\frac{|Q|}{|Q_0|}\geq\sum_{l(Q)\leq L}\frac{|Q|}{|Q_0|}\geq t,$$
    meaning $L^*\in\mathscr{S}(t,Q_0,E)$, which proves that $\mathscr{S}(t,Q_0,E)$ is an interval containing $[l(\mathcal{M}(Q_0)),\infty)$. Now, to see that $L'_t:=\inf\mathscr{S}(t,Q_0,E)>0$ for any fixed $t\in(0,1)$, note that the fact that $\sum_{l(Q)\leq l(\mathcal{M}(Q_0))}\frac{|Q|}{|Q_0|}=1$ implies that $\sum_{l(Q)\leq L}\frac{|Q|}{|Q_0|}\to0$ as $L\to0^+$, so we can pick $L>0$ sufficiently small as to make $\sum_{l(Q)\leq L}\frac{|Q|}{|Q_0|}<t$ and thus $L'_t>L>0$. To see that $\mathscr{S}(t,Q_0,E)$ contains its lower endpoint $L'_t$, we start by observing that
    $$\mathscr{S}(t,Q_0,E)\cap\{2^{-k}l(\mathcal{M}(Q_0))\}_{k\in\mathbb{N}_0}=\{2^{-k}l(\mathcal{M}(Q_0))\}_{k\leq k_0}$$ for some nonnegative integer $k_0$. We claim that $2^{-k_0}l(\mathcal{M}(Q_0))=L'_t$. Indeed, suppose that $2^{-k_0}l(\mathcal{M}(Q_0))>L\in\mathscr{S}(t,Q_0,E)$, for some $L>0$.  Then, $L>2^{-k_0-1}l(\mathcal{M}(Q_0))$, otherwise that would imply that $2^{-k_0-1}l(\mathcal{M}(Q_0))\in\mathscr{S}(t,Q_0,E)$, contradicting the hypothesis on $k_0$. Since there are no dyadic cubes with side lengths contained in the interval $(2^{-k_0-1}l(\mathcal{M}(Q_0)),2^{-k_0}l(\mathcal{M}(Q_0)))$, we have
    $$\sum_{l(Q)\leq 2^{-k_0-1}l(\mathcal{M}(Q_0))}\frac{|Q|}{|Q_0|}=\sum_{l(Q)\leq L}\frac{|Q|}{|Q_0|}\geq t$$
    implying $2^{-k_0-1}l(\mathcal{M}(Q_0))\in\mathscr{S}(t,Q_0,E)$, which, once again, is a contradiction. Thus, $L'_t=2^{-k_0}l(\mathcal{M}(Q_0))=\min\mathscr{S}(t,Q_0,E)$. 

    Finally, to prove (iii) simply observe that if $s\leq t$ and $L\in\mathscr{L}(t,Q_0,E)$, then $\sum_{l(Q)\geq L}\frac{|Q|}{|Q_0|}\geq t\geq s$, so $L\in\mathscr{L}(s,Q_0,E)$, and if $L\in\mathscr{S}(t,Q_0,E)$ then $\sum_{l(Q)\leq L}\frac{|Q|}{|Q_0|}\geq t\geq s$, implying $L\in\mathscr{S}(s,Q_0,E)$.
\end{proof}

\begin{definition}
    \label{def3}
    Fixed a non-empty set $E\subset\mathbb{R}^n$ with $|\overline{E}|=0$ and a cube $Q$ intersecting $\overline{E}$, we introduce the functions $\mathcal{L}_Q:(0,1)\to(0,l(\mathcal{M}(Q))]$ and $\mathcal{S}_Q:(0,1)\to(0,l(\mathcal{M}(Q))]$ given by
    $$\mathcal{L}_Q(t):=\max\mathscr{L}(t,Q,E)$$
    and
    $$\mathcal{S}_Q(t):=\min\mathscr{S}(t,Q,E).$$
\end{definition}

\begin{remark}
    Not only are the previous functions well-defined by Proposition~\ref{prop3}, but we also know that $\mathcal{L}_Q$ is non-increasing while $\mathcal{S}_Q$ is non-decreasing. Moreover, these functions can only take values of the form $2^{-k} l(\mathcal{M}(Q))$, although we will not usually make use of this fact.
\end{remark}

Intuitively, the function \( \mathcal{L}_Q \) (respectively, \( \mathcal{S}_Q \)) evaluated at \( t \) returns the smallest (largest) side length among those associated with the largest (smallest) pores that together make up a fraction \( t \) of the total volume of \( Q \setminus \overline{E} \). These functions will be essential for the characterization of $A_p$ sets in the following sections. But before that, we proceed to prove some basic properties regarding $\mathcal{L}_Q$ and $\mathcal{S}_Q$ that will be needed later for this purpose.

\begin{proposition}
    \label{prop4}
    Let $E$ be a closed non-empty set with $|E|=0$. If $Q_0\subset\mathbb{R}^n$ is a cube intersecting $E$, then for every $0<t<1$,
    \begin{equation}
        \label{eq4}
        \sum_{l(Q)\geq\mathcal{L}_{Q_0}(t)}\frac{|Q|}{|Q_0|}\geq t\hspace{0.5cm}\text{and}\hspace{0.5cm}\sum_{l(Q)\leq\mathcal{S}_{Q_0}(t)}\frac{|Q|}{|Q_0|}\geq t
    \end{equation}
    while
    \begin{equation}
        \label{eq5}
        \sum_{l(Q)<\mathcal{L}_{Q_0}(t)}\frac{|Q|}{|Q_0|}\leq1-t\hspace{0.5cm}\text{and}\hspace{0.5cm}\sum_{l(Q)>\mathcal{S}_{Q_0}(t)}\frac{|Q|}{|Q_0|}\leq1-t.
    \end{equation}
\end{proposition}

\begin{proof}
    The two inequalities in \eqref{eq4} trivially follow from the fact that $\mathcal{L}_{Q_0}(t)\in\mathscr{L}(t,Q_0,E)$ and $\mathcal{S}_{Q_0}(t)\in\mathscr{S}(t,Q_0,E)$. The ones in \eqref{eq5} involve measures of complementary sets to ones in \eqref{eq4}, so simply observe that
    $$\sum_{l(Q)<\mathcal{L}_{Q_0}(t)}\frac{|Q|}{|Q_0|}=1-\sum_{l(Q)\geq\mathcal{L}_{Q_0}(t)}\frac{|Q|}{|Q_0|}\leq1-t$$
    and
    $$\sum_{l(Q)>\mathcal{S}_{Q_0}(t)}\frac{|Q|}{|Q_0|}=1-\sum_{l(Q)\leq\mathcal{S}_{Q_0}(t)}\frac{|Q|}{|Q_0|}\leq1-t.$$
\end{proof}

\begin{lemma}
    \label{lemma4}
    Given a non-empty closed set $E$ with $|E|=0$ and some cube $Q_0$ such that $Q_0\cap E\neq\emptyset$, then
    \begin{itemize}
        \item[(i)] $\mathcal{L}_{Q_0}(t)\in\mathscr{S}(1-t,Q_0,E)$ for every $t\in(0,1)$. In particular, $\mathcal{L}_{Q_0}(t)\geq\mathcal{S}_{Q_0}(1-t)$ for $0<t<1$;
        \item[(ii)] $\mathcal{L}_{Q_0}$ and $\mathcal{S}_{Q_0}$ have the limit behaviors
        $$\lim_{t\to1^-}\mathcal{L}_{Q_0}(t)=\lim_{t\to0^+}\mathcal{S}_{Q_0}(t)=0$$
        and
        $$\lim_{t\to0^+}\mathcal{L}_{Q_0}(t)=\lim_{t\to1^-}\mathcal{S}_{Q_0}(t)=l(\mathcal{M}(Q_0)).$$
    \end{itemize}
\end{lemma}

\begin{proof}
    Fix a cube $Q_0$ as stated and take $t\in(0,1)$. We have that
    \begin{align*}
        \Big| \bigcup_{l(Q)>\mathcal{L}_{Q_0}(t)} Q\Big| = \Big| \bigcup_{\varepsilon>0} \Big( \bigcup_{l(Q)\geq\mathcal{L}_{Q_0}(t)+\varepsilon}Q \Big)\Big|=\lim_{\varepsilon\to0^+}\sum_{l(Q)\geq\mathcal{L}_{Q_0}(t)+\varepsilon}|Q|\leq t|Q_0|,
    \end{align*}
    where the last inequality holds because $\mathcal{L}_{Q_0}(t)=\max\mathscr{L}(t,Q_0,E)$. It follows that
    \begin{align*}
        \sum_{l(Q)\leq\mathcal{L}_{Q_0}(t)}\frac{|Q|}{|Q_0|}=1-\sum_{l(Q)>\mathcal{L}_{Q_0}(t)}\frac{|Q|}{|Q_0|}\geq1-t,
    \end{align*}
    thus $\mathcal{L}_{Q_0}(t)\in\mathscr{S}(1-t,Q_0,E)$. This proves (i). On the other hand, since $\mathcal{L}_{Q_0}$ is non-increasing in $(0,1)$, $\lim_{t\to1^-}\mathcal{L}_{Q_0}(t)=L^*$ for some $L^*\geq0$. This means that $\sum_{l(Q)\geq L^*}\frac{|Q|}{|Q_0|}\geq t$, for every $0<t<1$, implying $\sum_{l(Q)\geq L^*}\frac{|Q|}{|Q_0|}=1$. If we had $L^*>0$, since $\mathcal{D}_E(Q_0)$ is an infinite family, we could pick $Q_1\in\mathcal{D}_E(Q_0)$ such that $l(Q_1)<L^*$ and, in consequence,
    \begin{align*}
        \sum_{l(Q)\geq L^*}\frac{|Q|}{|Q_0|}=1-\sum_{l(Q)<L^*}\frac{|Q|}{|Q_0|}\leq1-\frac{|Q_1|}{|Q_0|}<1.
    \end{align*}
    Therefore $L^*=0$ and, by item (i),
    $$\lim_{t\to0^+}\mathcal{S}_{Q_0}(t)\leq\lim_{t\to0^+}\mathcal{L}_{Q_0}(1-t)=\lim_{t\to1^-}\mathcal{L}_{Q_0}(t)=0.$$
    Next, if we let $s:=\sum_{l(Q)\leq\frac{1}{2}l(\mathcal{M}(Q_0))}\frac{|Q|}{|Q_0|}$, then $s\in(0,1)$ and $\mathcal{S}_{Q_0}(t)=l(\mathcal{M}(Q_0))$ for every $s<t<1$. Finally, using (i) again we get
    $$l(\mathcal{M}(Q_0))\geq\lim_{t\to0^+}\mathcal{L}_{Q_0}(t)\geq\lim_{t\to0^+}\mathcal{S}_{Q_0}(1-t)=\lim_{t\to1^-}\mathcal{S}_{Q_0}(t)=l(\mathcal{M}(Q_0)).$$
\end{proof}

\section{First characterization of $A_p$ sets}
\label{sec4}

We now proceed to propose a first characterization of sets \( E \) for which there exists \( \alpha > 0 \) such that \( \dist(\cdot, E)^{-\alpha} \in A_p \), for \( 1 < p < \infty \). The two implications that constitute the equivalence to be proved are stated separately in Theorems~\ref{teo3} and~\ref{teo4}, and involve the functions \( \mathcal{L}_{Q_0} \) and \( \mathcal{S}_{Q_0} \), illustrating how effective these are in capturing the nature of \( A_p \) sets.

\begin{theorem}
    \label{teo3}
    Let $E$ be a closed non-empty set and suppose there exists $\alpha>0$ such that $\dist(\cdot,E)^{-\alpha}\in A_p$ for some $1<p<\infty$. Then, there exist positive constants $\sigma=\sigma(\alpha,p)$ and $C_0=C_0(n,\alpha,p,[\dist(\cdot,E)^{-\alpha}]_{A_p})$ such that
    \begin{equation}
    \label{eq7}
        \mathcal{L}_{Q_0}(t)\leq\frac{C_0}{t^\sigma}\mathcal{S}_{Q_0}(t)
    \end{equation}
    for every cube $Q_0$ intersecting $E$ and every $t\in(0,1)$.
\end{theorem}

\begin{proof}
    If $w(x)=\dist(x,E)^{-\alpha}$, since $w\in A_p$, in particular $w$ is locally integrable. But $w(x)=\infty$ for every $x\in E$, so we must have $|E|=0$, thus the functions $\mathcal{L}_{Q_0}$ and $\mathcal{S}_{Q_0}$ are well-defined. Let $t\in(0,1)$. Making use of the $A_p$ condition of $w$ over $Q_0$ as well as Lemma \ref{lemma2} and the properties of $\mathcal{L}_{Q_0}$ and $\mathcal{S}_{Q_0}$,
    \begin{align*}
        [\dist(\cdot,E)^{-\alpha}]_{A_p}&\geq\Big(\fint_{Q_0}\dist(x,E)^{-\alpha}dx\Big)\Big(\fint_{Q_0}\dist(x,E)^{\frac{\alpha}{p-1}}dx\Big)^{p-1} \\
        &\geq C(n,\alpha,p)\Big(\sum_{Q\in\mathcal{D}_E(Q_0)}l(Q)^{-\alpha}\frac{|Q|}{|Q_0|}\Big)\Big(\sum_{Q\in\mathcal{D}_E(Q_0)}l(Q)^{\frac{\alpha}{p-1}}\frac{|Q|}{|Q_0|}\Big)^{p-1}
        \\
        &\geq C(n,\alpha,p)\Bigg(\sum_{\substack{Q\in\mathcal{D}_E(Q_0)\\l(Q)\leq \mathcal{S}_{Q_0}(t)}}l(Q)^{-\alpha}\frac{|Q|}{|Q_0|}\Bigg)\Bigg(\sum_{\substack{Q\in\mathcal{D}_E(Q_0)\\l(Q)\geq \mathcal{L}_{Q_0}(t)}}l(Q)^{\frac{\alpha}{p-1}}\frac{|Q|}{|Q_0|}\Bigg)^{p-1}
        \\
        &\geq C(n,\alpha,p)\mathcal{S}_{Q_0}(t)^{-\alpha}\Bigg(\sum_{\substack{Q\in\mathcal{D}_E(Q_0)\\l(Q)\leq\mathcal{S}_{Q_0}(t)}}\frac{|Q|}{|Q_0|}\Bigg)\mathcal{L}_{Q_0}(t)^\alpha\Bigg(\sum_{\substack{Q\in\mathcal{D}_E(Q_0)\\l(Q)\geq \mathcal{L}_{Q_0}(t)}}\frac{|Q|}{|Q_0|}\Bigg)^{p-1}
        \\
        &\geq C(n,\alpha,p)t^{p}\mathcal{S}_{Q_0}(t)^{-\alpha}\mathcal{L}_{Q_0}(t)^\alpha,
    \end{align*}
    where the last inequality is a consequence of \eqref{eq4}. From this, we can easily get \eqref{eq7} with $\sigma=\frac{p}{\alpha}$ and $C_0$ depending on $n,\alpha,p$ and $[\dist(\cdot,E)^{-\alpha}]_{A_p}$.
\end{proof}

The converse to the above statement is given by the next result.

\begin{theorem}
    \label{teo4}
    Let $E$ be a closed non-empty set with $|E|=0$ such that there exist constants $\sigma,C_0>0$ satisfying
    \begin{equation*}
        \mathcal{L}_{Q_0}(t)\leq\frac{C_0}{t^\sigma}\mathcal{S}_{Q_0}(t)
    \end{equation*}
    for every cube $Q_0$ intersecting $E$ and every $t\in(0,1)$.
    Then, for every $1<p<\infty$ we have that $\dist(\cdot,E)^{-\alpha}\in A_p$ whenever $0<\alpha<\min\{\sigma^{-1},(p-1)\sigma^{-1}\}$.
\end{theorem}

\begin{proof}
    Pick $1<p<\infty$. We look for constants $\alpha,C>0$ such that the inequality 
    $$\Big(\fint_{Q_0}\dist(x,E)^{-\alpha}dx\Big)\Big(\fint_{Q_0}\dist(x,E)^{\frac{\alpha}{p-1}}dx\Big)^{p-1}\leq C$$
    holds for any cube $Q_0\subset\mathbb{R}^n$. To this end, we consider three cases separately. In the following, denote by $\dist(A,B):=\inf_{x\in A,y\in B}|x-y|$ the distance between two sets $A,B\subset\mathbb{R}^n$. Also, we let $\text{diam}(A):=\sup_{x,y\in A}|x-y|$ be the diameter of a set $A\subset\mathbb{R}^n$. \\
    
    \noindent\textbf{Case I} ($Q_0\cap E=\emptyset\land \dist(Q_0,E)\geq2\text{diam}(Q_0)$): in this case, it can be shown that if $x_0$ is some fixed point in $Q_0$, then $\frac{1}{2}\dist(x_0,E)\leq \dist(x,E)\leq2\dist(x_0,E)$ for every $x\in Q_0$. Thus, for any choice of $\alpha>0$ we have that $\fint_{Q_0}\dist(x,E)^{-\alpha}dx\leq2^\alpha \dist(x_0,E)^{-\alpha}$ and $\fint_{Q_0}\dist(x,E)^{\frac{\alpha}{p-1}}dx\leq2^{\frac{\alpha}{p-1}}\dist(x_0,E)^{\frac{\alpha}{p-1}}$, which means
    $$\Big(\fint_{Q_0}\dist(x,E)^{-\alpha}dx\Big)\Big(\fint_{Q_0}\dist(x,E)^{\frac{\alpha}{p-1}}dx\Big)^{p-1}\leq 4^\alpha.$$

    \noindent\textbf{Case II} ($Q_0\cap E=\emptyset\land\dist(Q_0,E)<2\text{diam}(Q_0)$): let us notice that $4Q_0\cap E\neq\emptyset$, where $4Q_0$ is the cube having the same center as $Q_0$ but such that $l(4Q_0)=4l(Q_0)$. Now assuming the $A_p$ condition holds for some $\alpha>0$ and with a constant $C>0$ for any cube intersecting $E$, we have
    \begin{align*}
        \fint_{Q_0}\dist(x,E)^{-\alpha}dx&\leq4^n\fint_{4Q_0}\dist(x,E)^{-\alpha}dx\\
        &\leq4^nC\Big(\fint_{4Q_0}\dist(x,E)^{\frac{\alpha}{p-1}}dx\Big)^{1-p}\leq C(n,p)\Big(\fint_{Q_0}\dist(x,E)^{\frac{\alpha}{p-1}}dx\Big)^{1-p}.
    \end{align*}
    Thus, proving this case amounts to proving the remaining Case III. \\
    
    \noindent\textbf{Case III} ($Q_0\cap E\neq\emptyset$): let $0<\alpha<1$ to be chosen later and begin by writing
    $$\mathcal{D}_E(Q_0)=\{Q:l(Q)\geq\mathcal{S}_{Q_0}(\tfrac{1}{2})\}\cup\bigcup_{k=1}^\infty\{Q:\mathcal{S}_{Q_0}(\tfrac{1}{2^{k+1}})\leq l(Q)<\mathcal{S}_{Q_0}(\tfrac{1}{2^k})\}.$$
    This equality is justified by the limit behavior $\lim_{t\to0^+}\mathcal{S}_{Q_0}(t)=0$ given by Lemma \ref{lemma4} (ii) and allows to make the estimate
    \begin{align*}
        \fint_{Q_0}\dist(x,E)^{-\alpha}dx\leq \hspace{11cm}\\ C(n,\alpha)\Big[\sum_{l(Q)\geq\mathcal{S}_{Q_0}(\frac{1}{2})}l(Q)^{-\alpha}\frac{|Q|}{|Q_0|}+\sum_{k=1}^\infty\Big(\sum_{\mathcal{S}_{Q_0}(\frac{1}{2^{k+1}})\leq l(Q)<\mathcal{S}_{Q_0}(\frac{1}{2^k})}l(Q)^{-\alpha}\frac{|Q|}{|Q_0|}\Big)\Big].
    \end{align*}
    We can find an appropriate bound for the first term inside the brackets making
    \begin{align*}
        \sum_{l(Q)\geq\mathcal{S}_{Q_0}(\frac{1}{2})}l(Q)^{-\alpha}\frac{|Q|}{|Q_0|}\leq\mathcal{S}_{Q_0}(\tfrac{1}{2})^{-\alpha}\leq\mathcal{L}_{Q_0}(\tfrac{1}{2})^{-\alpha},
    \end{align*}
    where Lemma \ref{lemma4} (i) was used to arrive at the last inequality. This same lemma together with \eqref{eq5} can be applied as follows to deduce a similar estimate for the remaining term:
    \begin{align*}
        \sum_{k=1}^\infty\Big(\sum_{\mathcal{S}_{Q_0}(\frac{1}{2^{k+1}})\leq l(Q)<\mathcal{S}_{Q_0}(\frac{1}{2^k})}l(Q)^{-\alpha}&\frac{|Q|}{|Q_0|}\Big)\hspace{10cm} \\
        &\leq\sum_{k=1}^\infty\mathcal{S}_{Q_0}(\tfrac{1}{2^{k+1}})^{-\alpha}\sum_{l(Q)<\mathcal{S}_{Q_0}(\frac{1}{2^k})}\frac{|Q|}{|Q_0|} \\
        &\leq C(\alpha,C_0)\sum_{k=1}^\infty2^{\alpha\sigma(k+1)}\mathcal{L}_{Q_0}(\tfrac{1}{2^{k+1}})^{-\alpha}\sum_{l(Q)<\mathcal{S}_{Q_0}(\frac{1}{2^k})}\frac{|Q|}{|Q_0|} \\
        &\leq C(\alpha,C_0)\mathcal{L}_{Q_0}(\tfrac{1}{2})^{-\alpha}\sum_{k=1}^\infty2^{\alpha\sigma(k+1)}\sum_{l(Q)<\mathcal{L}_{Q_0}(1-\frac{1}{2^k})}\frac{|Q|}{|Q_0|} \\
        &\leq C(\alpha,C_0)\mathcal{L}_{Q_0}(\tfrac{1}{2})^{-\alpha}\sum_{k=1}^\infty2^{\alpha\sigma(k+1)-k} \\
        &=C(\alpha,\sigma,C_0)\mathcal{L}_{Q_0}(\tfrac{1}{2})^{-\alpha}\sum_{k=1}^\infty2^{k(\alpha\sigma-1)}\\
        &=C(\alpha,\sigma,C_0)\mathcal{L}_{Q_0}(\tfrac{1}{2})^{-\alpha},
    \end{align*}
    if $\alpha<\sigma^{-1}$. We have shown that for these values of $\alpha$ the integral $\fint_{Q_0}d(x,E)^{-\alpha}dx$ is bounded by $\mathcal{L}_{Q_0}(\tfrac{1}{2})^{-\alpha}$ times some constant. All is left to show is the inequality $(\fint_{Q_0}d(x,E)^\frac{\alpha}{p-1} dx)^{p-1}\leq C\mathcal{L}_{Q_0}(\tfrac{1}{2})^\alpha$ for small enough values of $\alpha$. To this end we now consider the partition
    $$\mathcal{D}_E(Q_0)=\{Q:l(Q)\leq\mathcal{L}_{Q_0}(\tfrac{1}{2})\}\cup\bigcup_{k=1}^\infty\{Q:\mathcal{L}_{Q_0}(\tfrac{1}{2^k})<l(Q)\leq\mathcal{L}_{Q_0}(\tfrac{1}{2^{k+1}})\}.$$
    For sufficiently large \( k \), the sets \( \{\mathcal{L}_{Q_0}(\tfrac{1}{2^k})<l(Q)\leq\mathcal{L}_{Q_0}(\tfrac{1}{2^{k+1}})\} \) eventually become empty, but this does not compromise the proof. As before, we begin by splitting the considered integral as
    \begin{align*}
        \fint_{Q_0}\dist(x,E)^{\frac{\alpha}{p-1}}dx\leq \hspace{10.9cm}\\ C(n,\alpha,p)\Big[\sum_{l(Q)\leq\mathcal{L}_{Q_0}(\frac{1}{2})}l(Q)^{\frac{\alpha}{p-1}}\frac{|Q|}{|Q_0|}+\sum_{k=1}^\infty\Big(\sum_{\mathcal{L}_{Q_0}(\frac{1}{2^k})< l(Q)\leq\mathcal{L}_{Q_0}(\frac{1}{2^{k+1}})}l(Q)^{\frac{\alpha}{p-1}}\frac{|Q|}{|Q_0|}\Big)\Big].
    \end{align*}
    The first term can easily be shown to be less than or equal to $\mathcal{L}_{Q_0}(\frac{1}{2})^{\frac{\alpha}{p-1}}$. For the second term, we once again rely on Lemma \ref{lemma4} and Proposition \ref{prop4} to make the estimate
    \begin{align*}
        \sum_{k=1}^\infty\Big(\sum_{\mathcal{L}_{Q_0}(\frac{1}{2^k})< l(Q)\leq\mathcal{L}_{Q_0}(\frac{1}{2^{k+1}})}&l(Q)^{\frac{\alpha}{p-1}}\frac{|Q|}{|Q_0|}\Big) \\
        &\leq\sum_{k=1}^\infty\mathcal{L}_{Q_0}(\tfrac{1}{2^{k+1}})^{\frac{\alpha}{p-1}}\sum_{l(Q)>\mathcal{L}_{Q_0}(\frac{1}{2^k})}\frac{|Q|}{|Q_0|} \\
        &\leq C(\alpha,p,C_0)\sum_{k=1}^\infty2^{\frac{\alpha\sigma}{p-1}(k+1)}\mathcal{S}_{Q_0}(\tfrac{1}{2^{k+1}})^{\frac{\alpha}{p-1}}\sum_{l(Q)>\mathcal{L}_{Q_0}(\frac{1}{2^k})}\frac{|Q|}{|Q_0|} \\
        &\leq C(\alpha,p,C_0)\mathcal{S}_{Q_0}(\tfrac{1}{2})^{\frac{\alpha}{p-1}}\sum_{k=1}^\infty2^{\frac{\alpha\sigma}{p-1}(k+1)}\sum_{l(Q)>\mathcal{S}_{Q_0}(1-\frac{1}{2^k})}\frac{|Q|}{|Q_0|} \\
        &\leq C(\alpha,p,\sigma,C_0)\mathcal{S}_{Q_0}(\tfrac{1}{2})^{\frac{\alpha}{p-1}}\sum_{k=1}^\infty2^{k(\frac{\alpha\sigma}{p-1}-1)} \\
        &\leq C(\alpha,p,\sigma,C_0)\mathcal{L}_{Q_0}(\tfrac{1}{2})^{\frac{\alpha}{p-1}}
    \end{align*}
    taking $\alpha<(p-1)\sigma^{-1}$, from which $(\fint_{Q_0}d(x,E)^\frac{\alpha}{p-1} dx)^{p-1}\leq C\mathcal{L}_{Q_0}(\tfrac{1}{2})^\alpha$ is deduced.
\end{proof}

\section{Second characterization of $A_p$ sets}
\label{sec5}

Theorems \ref{teo3} and \ref{teo4} combined provide a first characterization of $A_p$ sets based on an inequality that involves functions $\mathcal{L}$ and $\mathcal{S}$ evaluated at every $0<t<1$. In this section, we aim to demonstrate that a new characterization is possible where these functions are evaluated only at a single point, thus proving Theorem \ref{teo0}. To this end, let us start by remarking that \eqref{eq0} can be rewritten as
\begin{equation}
    \label{eq9}
    0<\frac{\sup\mathscr{L}(s,Q_0,E)}{\inf\mathscr{S}(s,Q_0,E)}\leq C_0
\end{equation}
for every cube $Q_0$ with $Q_0\cap\overline{E}\neq\emptyset$. A set $E$ for which this inequality holds for every appropriate cube will necessary have zero measure, so we can drop this extra hypothesis on upcoming results.

\begin{proposition}
    \label{prop5}
    Let $E\subset\mathbb{R}^n$ be a closed non-empty set such that there exist constants $0<s<1$ and $C_0>0$ for which \eqref{eq9} holds for every cube $Q_0$ intersecting $E$. Then, $|E|=0$.
\end{proposition}

\begin{proof}
    Suppose that $Q_0$ is a cube intersecting $E$. Following the usual convention of assigning the values $\sup\Omega=0$ and $\inf\Omega=\infty$ whenever the set $\Omega\subset\mathbb{R}_{\geq0}$ is empty, the positivity of the fraction in \eqref{eq9} implies that both $\mathscr{L}(s,Q_0,E)$ and $\mathscr{S}(s,Q_0,E)$ are non-empty sets. Thus, we can pick any $L\in\mathscr{L}(s,Q_0,E)$ and write
    $$|E\cap Q_0|=|Q_0\setminus (\mathbb{R}^n\setminus E)|\leq\Big|Q_0\setminus\bigcup_{\substack{Q\in\mathcal{D}_E(Q_0)\\l(Q)\geq L}}Q\Big|\leq(1-s)|Q_0|.$$
    Let us observe that the bound $|E\cap Q_0|\leq(1-s)|Q_0|$ also holds trivially in the case where $Q_0\cap E=\emptyset$. Thus, if $\mathbbm{1}_E$ denotes the characteristic function of $E$ and $Q(x,r)$ is the cube centered at $x$ and with $l(Q)=2r$, then for almost every $x\in\mathbb{R}^n$ we have that
    \begin{align*}
        \mathbbm{1}_E(x)=\lim_{r\to0^+}\fint_{Q(x,r)}\mathbbm{1}_E(y)dy=\lim_{r\to0^+}\frac{|E\cap Q(x,r)|}{|Q(x,r)|}\leq(1-s)<1.
    \end{align*}
    Therefore, $|E|=0$.
\end{proof}

By Proposition \ref{prop5} and in view of Definition \ref{def3}, \eqref{eq0} and \eqref{eq9} are both equivalent to the condition
\begin{equation}
    \label{eq10}
    \mathcal{L}_{Q_0}(s)\leq C_0\:\mathcal{S}_{Q_0}(s)
\end{equation}
for every cube $Q_0$ with $Q_0\cap\overline{E}\neq\emptyset$. The validity of \eqref{eq10} for a sufficiently small value of $s$ is the condition to be shown equivalent to the first characterization of $A_p$ sets. Although, we will need to prove first a technical lemma before demonstrating this.

\begin{lemma}
    \label{lemma5}
    Given a closed set $E$ with $|E|=0$, the following conditions are equivalent.
    \begin{itemize}
        \item[(i)] There exist constants $\sigma,C_0>0$ such that for every cube $Q_0$ intersecting $E$ and every $t\in(0,1)$,
        \begin{equation*}
        \mathcal{L}_{Q_0}(t)\leq\frac{C_0}{t^\sigma}\mathcal{S}_{Q_0}(t).
    \end{equation*}
        \item[(ii)] There exist constants $\sigma,C_0'>0$ and $0<t'<\tfrac{1}{2}$ such that for every cube $Q_0$ intersecting $E$ and every $0<t\leq t'$,
        \begin{equation*}
        \mathcal{L}_{Q_0}(t)\leq\frac{C_0'}{t^\sigma}\mathcal{S}_{Q_0}(t).
        \end{equation*}
    \end{itemize}
\end{lemma}

\begin{proof}
    (i)$\implies$(ii) is immediate. To see that (ii)$\implies$(i), let $t'<t<1$, then 
    $$\mathcal{L}_{Q_0}(t)\leq\mathcal{L}_{Q_0}(t')\leq\frac{C_0'}{(t')^\sigma}\mathcal{S}_{Q_0}(t')\leq\frac{C_0'}{(tt')^\sigma}\mathcal{S}_{Q_0}(t')\leq\frac{C_0}{t^\sigma}\mathcal{S}_{Q_0}(t)$$
    where $C_0:=C_0'(t')^{-\sigma}$.
\end{proof}

From now on, if $Q\in\mathcal{D}^k(Q_0)$, we set $\pi_0Q:=Q$ and introduce the notation $\pi_j Q:=\pi(\pi_{j-1}Q)$ for $1\leq j\leq k$. In particular, $\pi_k Q=Q_0$. Also, let us make the following important remark. If $Q_0\cap\overline{E}\neq\emptyset$ and $Q\in\mathcal{D}(Q_0)$ with $Q\cap\overline{E}\neq\emptyset$ as well, then $\mathcal{D}_E(Q)\subset\mathcal{D}_E(Q_0)$, this is, the grids are compatible.

\begin{lemma}
    \label{lemmaa}
    Let $E\subset\mathbb{R}^n$ be a closed non-empty set. Then, the next statements are equivalent.
    \begin{itemize}
        \item[(i)] $|E|=0$ and there exist constants $\Tilde{\sigma},\Tilde{C}_0>0$ satisfying
        $$\mathcal{L}_{Q_0}(t)\leq\frac{\Tilde{C}_0}{t^{\Tilde{\sigma}}}\mathcal{S}_{Q_0}(t)$$
        for every cube $Q_0$ intersecting $E$ and every $t\in(0,1)$.
        \item[(ii)] There exist $C_0>0$ and $0<s<(1+2^n)^{-1}$ such that \eqref{eq10} holds for every cube $Q_0\subset\mathbb{R}^n$ intersecting $E$.
    \end{itemize}
\end{lemma}

\begin{proof}
    Clearly, (i) implies (ii), so we turn our attention to the remaining implication. Fix then a cube $Q_0$ in $\mathbb{R}^n$ and let $0<t\leq s$. After all, by Lemma \ref{lemma5}, it is enough to prove the inequality $\mathcal{L}_{Q_0}(t)\leq\frac{C_0}{t^\sigma}\mathcal{S}_{Q_0}(t)$ for values of $t$ in $(0,s]$. According to our hypothesis, $\mathcal{L}_{Q_0}(s)\leq C_0\:\mathcal{S}_{Q_0}(s)$. But since $\mathcal{L}_{Q_0}(\tfrac{1}{2})\geq\mathcal{S}_{Q_0}(\frac{1}{2})$ by Lemma \ref{lemma4} and recalling that $\mathcal{L}_{Q_0}$ is non-increasing while $\mathcal{S}_{Q_0}$ is non-decreasing, then $\mathcal{L}_{Q_0}(s)\geq\mathcal{S}_{Q_0}(s)$ and necessarily $C_0\geq1$.
    
    We claim that $\mathcal{S}_{Q_0}(s)\leq C_0(\frac{s}{t})^\sigma\mathcal{S}_{Q_0}(t)$ for some $\sigma>0$. If $\mathcal{S}_{Q_0}(t)=\mathcal{S}_{Q_0}(s)$, there is nothing to prove (as we just saw that $C_0\geq1$), so let us assume without loss of generality that $\mathcal{S}_{Q_0}(t)<\mathcal{S}_{Q_0}(s)$. Notice that
    \begin{equation}
        \label{eqa}
        t\leq\sum_{l(Q)\leq\mathcal{S}_{Q_0}(t)}\frac{|Q|}{|Q_0|}<s
    \end{equation}
    where the second inequality comes from the fact that $\mathcal{S}_{Q_0}(t)<\mathcal{S}_{Q_0}(s)=\min\mathscr{S}(s,Q_0,E)$ and so $\mathcal{S}_{Q_0}(t)\notin\mathscr{S}(s,Q_0,E)$.

    Let $\{Q_\alpha\}_{\alpha\in A}=\{Q\in\mathcal{D}_E(Q_0):l(Q)\leq\mathcal{S}_{Q_0}(t)\}$ for some appropriate index set $A$. Given some $\alpha\in A$, $Q_\alpha\in\mathcal{D}^k(Q_0)$ for some generation $k=k(\alpha)$, so we can define 
    $$J_\alpha:=\max\Big\{j\in\{0,\hdots,k\}:\frac{|\pi_jQ_\alpha\cap\bigcup_{\beta\in A} Q_\beta|}{|\pi_jQ_\alpha|}\geq s\Big\}.$$
    Since $|\pi_k Q_\alpha\cap\bigcup_{\beta\in A}Q_\beta||\pi_k Q_\alpha|^{-1}=|\bigcup_{\beta\in A}Q_\beta||Q_0|^{-1}<s$ by \eqref{eqa}, clearly $J_\alpha<k$. Furthermore, $|\pi_0 Q_\alpha\cap\bigcup_{\beta\in A}Q_\beta||\pi_0 Q_\alpha|^{-1}=|Q_\alpha||Q_\alpha|^{-1}=1$ and $|\pi_1 Q_\alpha\cap\bigcup_{\beta\in A}Q_\beta||\pi_1 Q_\alpha|^{-1}\geq|Q_\alpha||\pi Q_\alpha|^{-1}=\frac{1}{2^n}\geq s$, so $J_\alpha\geq1$. In particular, $\pi Q_\alpha\subset \pi_{J_\alpha}Q_\alpha$ and thus $\pi_{J_\alpha}Q_\alpha\cap E\neq\emptyset$. We can also see that, since $\pi_{J_\alpha+1}Q_\alpha$ is well-defined,
    \begin{align*}
        \frac{|\pi_{J_\alpha}Q_\alpha\cap\bigcup_\beta Q_\beta|}{|\pi_{J_\alpha}Q_\alpha|}\leq\frac{|\pi_{J_\alpha+1}Q_\alpha\cap\bigcup_\beta Q_\beta|}{|\pi_{J_\alpha}Q_\alpha|}=2^n\frac{|\pi_{J_\alpha+1}Q_\alpha\cap\bigcup_\beta Q_\beta|}{|\pi_{J_{\alpha+1}}Q_\alpha|}<2^ns.
    \end{align*}
    Thus, $|\pi_{J_\alpha}Q_\alpha\cap\bigcup_\beta Q_\beta|<2^ns|\pi_{J_\alpha}Q_\alpha|$. Denoting $P_\alpha:=\pi_{J_\alpha}Q_\alpha$, since $P_\alpha\cap P_\beta\neq\emptyset$ if and only if $P_\alpha\subset P_\beta$ or $P_\beta\subset P_\alpha$, we can find another index set $\Gamma$ such that $\bigcup_{\alpha\in\Gamma}P_\alpha=\bigcup_{\alpha\in A}P_\alpha$ and $P_\alpha\cap P_\beta=\emptyset$ for every pair of indexes $\alpha,\beta\in\Gamma$. In other words, $\Gamma$ is chosen in order to have that $\{P_\alpha\}_{\alpha\in\Gamma}$ is the maximal subcollection of $\{P_\alpha\}_{\alpha\in A}$. Using the pairwise disjoint property of the collection $\{P_\alpha\}_{\alpha\in\Gamma}$, we find that
    \begin{equation}
        \label{eqa}
        \Big|\bigcup_{\beta\in A}Q_\beta\Big|=\Big|\bigcup_{\alpha\in\Gamma}P_\alpha\cap\bigcup_{\beta\in A}Q_\beta\Big|=\sum_{\alpha\in\Gamma}\Big|P_\alpha\cap\bigcup_{\beta\in A}Q_\beta\Big|<2^ns\sum_{\alpha\in\Gamma}|P_\alpha|=2^ns\Big|\bigcup_{\alpha\in\Gamma}P_\alpha\Big|. 
    \end{equation}
    Now, for a fixed $\alpha\in\Gamma$ we obtain
    \begin{align*}
        \sum_{\substack{Q\in\mathcal{D}_E(P_\alpha)\\l(Q)\leq \mathcal{S}_{Q_0}(t)}}\frac{|Q|}{|P_\alpha|}=\sum_{\beta\in A:Q_\beta\subset P_\alpha}\frac{|Q_\beta|}{|P_\alpha|}=\frac{|P_\alpha\cap\bigcup_{\beta\in A}Q_\beta|}{|P_\alpha|}\geq s
    \end{align*}
    by definition of $P_\alpha$. But then, $\mathcal{S}_{Q_0}(t)\in\mathscr{S}(s,P_\alpha,E)$ and thus $\mathcal{S}_{Q_0}(t)\geq\mathcal{S}_{P_\alpha}(s)\geq C_0^{-1}\mathcal{L}_{P_\alpha}(s)\geq C_0^{-1}\mathcal{S}_{P_\alpha}(1-s)$ by means of (ii) applied to the cube $P_\alpha$. This implies $C_0\mathcal{S}_{Q_0}(t)\in\mathscr{S}(1-s,P_\alpha,E)$ and
    $$\sum_{\substack{Q\in\mathcal{D}_E(P_\alpha)\\l(Q)\leq C_0\mathcal{S}_{Q_0}(t)}}\frac{|Q|}{|P_\alpha|}\geq1-s.$$

    Observing that $\bigcup_{\alpha\in\Gamma}\mathcal{D}_E(P_\alpha)\subset\mathcal{D}_E(Q_0)$ and making use of \eqref{eqa}, we find that
    \begin{align*}
        \sum_{\substack{Q\in\mathcal{D}_E(Q_0)\\l(Q)\leq C_0\mathcal{S}_{Q_0}(t)}}\frac{|Q|}{|Q_0|}\geq\sum_{\alpha\in\Gamma}\frac{|P_\alpha|}{|Q_0|}\sum_{\substack{Q\in\mathcal{D}_E(P_\alpha)\\l(Q)\leq C_0\mathcal{S}_{Q_0}(t)}}\frac{|Q|}{|P_\alpha|}&\geq(1-s)\sum_{\alpha\in\Gamma}\frac{|P_\alpha|}{|Q_0|} \\
        &=(1-s)\frac{|\bigcup_{\alpha\in\Gamma}P_\alpha|}{|Q_0|} \\
        &\geq\frac{(1-s)}{2^ns}\frac{|\bigcup_{\alpha\in A}Q_\alpha|}{|Q_0|} \\
        &=\frac{(1-s)}{2^ns}\sum_{\substack{Q\in\mathcal{D}_E(Q_0)\\l(Q)\leq \mathcal{S}_{Q_0}(t)}}\frac{|Q|}{|Q_0|} \\
        &\geq\frac{(1-s)t}{2^ns}.
    \end{align*}

    This last inequality implies that $\mathcal{S}_{Q_0}(t)\geq C_0^{-1}\mathcal{S}_{Q_0}(ct)$, where $c:=\frac{(1-s)}{2^ns}$. Since $s<\frac{1}{2^n+1}$, the constant $c$ is strictly greater than one and $0<ct<1$ for every $t\leq s$. Let $m$ be the integer such that $c^{m-1}t<s\leq c^mt$ or, equivalently, $m-1<\frac{\log\frac{s}{t}}{\log c}\leq m$. Then,
    \begin{align*}
        \mathcal{S}_{Q_0}(t)\geq C_0^{-1}\mathcal{S}_{Q_0}(ct)\geq\hdots\geq C_0^{-m}\mathcal{S}_{Q_0}(c^mt)\geq C_0^{-m}\mathcal{S}_{Q_0}(s),
    \end{align*}
    this is,
    \begin{align*}
        \mathcal{S}_{Q_0}(s)\leq C_0^m\mathcal{S}_{Q_0}(t)\leq C_0^{1+\frac{\log\frac{s}{t}}{\log c}}\mathcal{S}_{Q_0}(t)=C_0\Big(\frac{s}{t}\Big)^{\frac{\log C_0}{\log c}}\mathcal{S}_{Q_0}(t).
    \end{align*}

    Notice that, if we had $C_0=1$, then the above computations show that $0<\mathcal{S}_{Q_0}(s)\leq\mathcal{S}_{Q_0}(t)\to0$ as $t\to0^+$, so $C_0$ must actually be strictly grater than $1$. Thus, the original claim is proved with $\sigma:=\frac{\log C_0}{\log c}>0$. Next, we claim that $\mathcal{L}_{Q_0}(t)\leq C_0(\frac{s}{t})^\sigma\mathcal{L}_{Q_0}(s)$ for $0<t\leq s$, again with $\sigma=\frac{\log C_0}{\log c}$, which can be proved in exactly the same way as the previous claim, and therefore the proof is omitted. Combining both claims, we get
    $$\mathcal{L}_{Q_0}(t)\leq C_0\Big(\frac{s}{t}\Big)^\sigma\mathcal{L}_{Q_0}(s)\leq C_0^2\Big(\frac{s}{t}\Big)^\sigma\mathcal{S}_{Q_0}(s)\leq C_0^3\Big(\frac{s}{t}\Big)^{2\sigma}\mathcal{S}_{Q_0}(t),$$
    which proves the theorem with $\Tilde{C}_0=C_0^3>0$ and $\Tilde{\sigma}=2\sigma>0$.
\end{proof}

Combining Lemma \ref{lemmaa} with some of our previous results, we can easily obtain a proof for the main result of this paper as stated in the introduction.

\begin{proof}[Proof of Theorem \ref{teo0}]
    By Proposition \ref{prop5}, (I) is equivalent to \eqref{eq10} for every cube $Q_0$ intersecting $\overline{E}$. Using Lemma \ref{lemmaa}, the latter is equivalent to the condition $|\overline{E}|=0$ and $\mathcal{L}_{Q_0}(t)\leq\frac{\Tilde{C}_0}{t^{\Tilde{\sigma}}}\mathcal{S}_{Q_0}(t)$ for every cube $Q_0$ intersecting $E$, every $t\in(0,1)$ and some constants $\Tilde{\sigma},\Tilde{C}_0>0$. Finally, Theorems \ref{teo3} and \ref{teo4} show this is equivalent to (II).
    
\end{proof}

\section{Further results: weak porosity and a probabilistic approach to pore distribution}
\label{sec6}

In this section, we aim to derive further consequences of the characterizations
of the \(A_p\) sets established beforehand.  First, we introduce the class of
weakly porous sets -known to coincide with $A_1$ sets- and demonstrate that they naturally meet these
characterizations.  We then analyze the distance weights associated to certain
subsets of the real line that have been introduced by other authors in recent
contributions. For this purpose, we use a probabilistic approach
that may prove valuable for future applications involving distance weights and
their associated sets in general.

\subsection{Weakly porous sets as $A_p$ sets}

The concept of weak porosity has been extended to general spaces such as metric spaces and spaces of homogeneous type, but the definition given in \cite{ANDERSON} in the context of Euclidean spaces will be more adequate for our purposes.

\begin{definition}
    \label{def1}
    Let $E \subset\mathbb{R}^n$ be a non-empty set. Then, $E$ is \textbf{weakly porous} if there are constants $0<\sigma,\gamma<1$ such that for all cubes $Q\subset \mathbb{R}^n$ there exist $N=N(Q)\in\mathbb{N}$ and pairwise disjoint $Q_i\in\mathcal{D}(Q)$, $i = 1,\dots,N$, satisfying $Q_i\subset Q\setminus\overline{E}$ with $l(Q_i)\geq\gamma l(\mathcal{M}(Q))$ for all $i=1,\dots,N$ and
    \[
        \sum_{i=1}^N |Q_i|\geq\sigma|Q|.
    \]
    
\end{definition} 

\begin{remark}
    \label{remark1}
    The original definition of weakly porous sets given in \cite{ANDERSON} requires each $Q_i\subset Q\setminus E$ instead of $Q_i\subset Q\setminus\overline{E}$. Nevertheless, since in that same reference it is shown that weak porosity is invariant under topological closure, both definitions can be seen to be equivalent. 
\end{remark}

Weakly porous sets in $\mathbb{R}^n$ have been shown to agree with the collection of sets spanning non-trivial $A_1$ distance weights, as the next theorem due to Anderson et al. states, and which should be compared with Theorem \ref{teo0}. In particular, because of the set inclusion $A_1\subset A_p$ for $p>1$, this means that weakly porous sets are $A_p$ sets, although the converse is not true, as the examples in Subsection \ref{subsec6.2} shows. In the following, we aim to obtain a more direct proof that weakly porous sets are $A_p$ sets by manually verifying \eqref{eq10} holds for any such set, hoping that while doing so, both conditions are shown to be naturally related. Specifically, \eqref{eq10} will be shown to be an even weaker condition than that of Definition \ref{def1}.

\begin{theorem}[\textup{\cite[Theorem 1.1]{ANDERSON}}]
    \label{teo1}
    Let $E\subset\mathbb{R}^n$ be a non-empty set. Then, the next statements are equivalent.
    \begin{itemize}
        \item[(i)] $E$ is weakly porous with  constants $0<\sigma,\gamma<1$;
        \item[(ii)] there exists $\alpha>0$ such that $\dist(\cdot,E)^{-\alpha}\in A_1$.
    \end{itemize}
\end{theorem}

The reason behind the strong relationship between weak porosity and $A_1$ distance weights may not be easy to recognize, but it primarily stems from the fact that the distribution of pores forming the complement of $Q\setminus\overline{E}$ depends heavily on the side length of $\mathcal{M}(Q)$ when $E$ is weakly porous. This is because, despite its relative simplicity, Definition \ref{def1} can be applied iteratively to a cube $Q$ and some of its subcubes in order to deduce this global behavior. In fact, this kind of iterative argument will be used next to show that, by choosing a smaller value for the parameter $\gamma$ in the definition of weak porosity, one can obtain an improvement in the value of $\sigma$ by a fixed amount. Before this, we reproduce a technical lemma that will be needed for that purpose.

\begin{lemma}[\textup{\cite[Lemma 3.2 (ii)]{ANDERSON}}]
    \label{lemma3}
    Let $E$ be a weakly porous set with constants $0<\sigma,\gamma<1$. Then, there exists a constant $C=C(\sigma,\gamma,n)>0$ such that if $Q,R\subset\mathbb{R}^n$ are two cubes satisfying $Q\subset R$ with $|R|=2^n|Q|$, then
    \begin{equation*}
        |\mathcal{M}(R)|\leq C|\mathcal{M}(Q)|.
    \end{equation*}
    In particular, the above inequality holds whenever $Q\in\mathcal{D}^1(R)$, this is, when $Q$ is a dyadic child of $R$.
\end{lemma}

\begin{theorem}
    \label{teo2}
    Let $E$ be a closed weakly porous set with constants $0<\sigma,\gamma<1$.
    \begin{itemize}
        \item[(i)] There exists $0<\widehat{\gamma}<\gamma$ depending on both $\sigma$ and $\gamma$ such that $E$ is weakly porous with constants $0<\widehat{\sigma},\widehat{\gamma}<1$, where $\widehat{\sigma}=1-(1-\sigma)^2>\sigma$.
        \item[(ii)] Let $\sigma^*$ be any number in between $0$ and $1$. Then, there exists $0<\gamma^*\leq\gamma$ with $\gamma^*=\gamma^*(\sigma^*,\sigma,\gamma)$ such that $E$ is weakly porous with constants $0<\sigma^*,\gamma^*<1$.
    \end{itemize}
\end{theorem}

\begin{proof}
    To prove (i), assume $E$ is as in the statement and let $Q_0\subset\mathbb{R}^n$ be some cube. Let us observe that the weak porosity condition is trivially fulfilled for any choice of parameters whenever $Q_0\cap E=\emptyset$, so assume that $Q_0$ intersects $E$. Fix $t\in(0,1)$ and consider the collection
    $$\mathcal{F}_t(Q_0)=\{Q\in\mathcal{D}_E(Q_0):l(Q)\geq t l(\mathcal{M}(Q_0))\}$$
    and its complementary family
    $$\mathcal{G}_t(Q_0)=\{Q\in\mathcal{D}(Q_0):Q\cap Q'=\emptyset\text{ for every } Q'\in\mathcal{F}_t(Q_0)\land \pi Q\cap Q'\neq\emptyset\text{ for some }Q'\in\mathcal{F}_t(Q_0)\}.$$
    These two families are mutually disjoint and their union is equal to $Q_0$ for any choice of $t$. Moreover, $l(Q)\geq tl(\mathcal{M}(Q_0))$ and $Q\cap E\neq\emptyset$ for every $Q\in\mathcal{G}_t(Q_0)$. It is easy to verify (see \cite{ANDERSON}) that our hypothesis of $E$ being weakly porous with constants $0<\sigma,\gamma<1$ is equivalent to the condition
    \begin{equation}
        \label{eq8}
        \sum_{Q\in\mathcal{F}_\gamma(Q')}|Q|\geq\sigma|Q'|,\text{ for every cube }Q'\text{ intersecting }E.
    \end{equation}
    We can take advantage of \eqref{eq8} by considering the following sums.
    \begin{align*}
        \sum_{Q\in\mathcal{F}_\gamma(Q_0)}|Q|+\sum_{R\in\mathcal{G}_\gamma(Q_0)}\sum_{Q\in\mathcal{F}_\gamma(R)}|Q|&\geq\sum_{Q\in\mathcal{F}_\gamma(Q_0)}|Q|+\sigma\sum_{R\in\mathcal{G}_\gamma(Q_0)}|R| \\
        &=\sum_{Q\in\mathcal{F}_\gamma(Q_0)}|Q|+\sigma\Big( |Q_0|-\sum_{Q\in\mathcal{F}_\gamma(Q_0)}|Q| \Big) \\
        &=(1-\sigma)\sum_{Q\in\mathcal{F}_\gamma(Q_0)}|Q|+\sigma|Q_0| \\
        &\geq\sigma(1-\sigma)|Q_0|+\sigma|Q_0| \\
        &=\sigma(2-\sigma)|Q_0| \\
        &=(1-(1-\sigma)^2)|Q_0|.
    \end{align*}
    
    On the other hand, if $R\in\mathcal{G}_\gamma(Q_0)$, then there exist $P\in\mathcal{G}_\gamma(Q_0)$ such that $\pi R\cap P\neq\emptyset$, which implies $P\subset\pi R$ (because $P\cap E=\emptyset$ while $R\cap E\neq\emptyset$). But then by Lemma \ref{lemma3}
    $$|\mathcal{M}(R)|\geq C(\sigma,\gamma,n)|\mathcal{M}(\pi R)|\geq C(\sigma,\gamma,n)|P|\geq C(\sigma,\gamma,n)|\mathcal{M}(Q_0)|.$$
    Thus, each $Q\in\mathcal{F}_\gamma(R)$ satisfies
    $$|Q|\geq\gamma^n|\mathcal{M}(R)|\geq\gamma^n C(\sigma,\gamma,n)|\mathcal{M}(Q_0)|=:(\widehat{\gamma})^n|\mathcal{M}(Q_0)|,$$
    implying that
    $$\Big(\bigcup_{Q\in\mathcal{F}_\gamma(Q_0)}Q\Big)\cup\Big(\bigcup_{R\in\mathcal{G}_\gamma(Q_0)}\bigcup_{Q\in\mathcal{F}_\gamma(R)}Q\Big)\subset\bigcup_{F_{\widehat{\gamma}}(Q_0)}Q$$
    and therefore $\sum_{F_{\widehat{\gamma}}(Q_0)}|Q|\geq (1-(1-\sigma)^2)|Q_0|$. Since the cube $Q_0$ intersecting $E$ was arbitrarily chosen, it follows that $E$ is weakly porous with constants $\widehat{\sigma}=1-(1-\sigma)^2$ and $\widehat{\gamma}$.
    
    To see (ii), fix $0<\sigma^*<1$ and let $\sigma_0=\sigma$ and $\gamma_0=\gamma$. If $\sigma^*\leq\sigma_0$, then it suffices to take $\gamma^*=\gamma_0$. If $\sigma^*>\sigma_0$, proceed to define, for each $k\in\mathbb{N}$, $\sigma_k=\widehat{\sigma}_{k-1}=1-(1-\sigma_{k-1})^2$ and $\gamma_k=\widehat{\gamma}_{k-1}$. This is, the pair $(\sigma_k,\gamma_k)$ is obtained by applying (i) to the weakly porous set $E$ having constants $(\sigma_{k-1},\gamma_{k-1})$. Defining $\eta_k=1-\sigma_k$, observe that the recursive relation for this quantity becomes $\eta_k=\eta_{k-1}^2=\eta_{k-2}^{2^2}=\hdots=\eta_0^{2^k}$. This means $\sigma_k=1-\eta_k=1-\eta_0^{2^k}=1-(1-\sigma)^{2^k}$. Choose $K=K(\sigma^*,\sigma)\in\mathbb{N}_0$ big enough as to make $\sigma_K\geq\sigma^*$. Since $E$ is weakly porous with constants $0<\sigma_K,\gamma_K<1$, then this also holds for the constants $0<\sigma^*,\gamma^*<1$ instead, where $\gamma^*=\gamma_K<\hdots<\gamma_0=\gamma$.
\end{proof}

Using this improving technique on the parameter $\sigma$ of a weakly porous set $E$, we can easily prove that \eqref{eq10} must hold in this case for an appropriate value of $C_0$.

\begin{theorem}
    \label{teo6}
    Let $E\subset\mathbb{R}^n$ be a closed weakly porous set with constants $0<\sigma,\gamma<1$ and let $s\in(0,\frac{1}{1+2^n})$. Then, there exists a constant $C_0>0$ dependent on $n$, $s$, $\sigma$ and $\gamma$ such that $\mathcal{L}_{Q_0}(s)\leq C_0\:\mathcal{S}_{Q_0}(s)$ for every cube $Q_0$ intersecting $E$.
\end{theorem}

\begin{proof}
    Use Theorem \ref{teo2} to choose $0<\sigma^*,\gamma^*<1$ such that $E$ is weakly porous with said constants and $1-\sigma^*<s$. Pick $\varepsilon>0$ and a cube $Q_0$ intersecting $E$ and observe that
    \begin{align*}
        \sum_{\substack{Q\in\mathcal{D}_E(Q_0)\\l(Q)\leq\gamma^* l(\mathcal{M}(Q_0))-\varepsilon}}|Q|&\leq\sum_{\substack{Q\in\mathcal{D}_E(Q_0)\\l(Q)<\gamma^* l(\mathcal{M}(Q_0))}}|Q| \\
        &=|Q_0|-\sum_{\substack{Q\in\mathcal{D}_E(Q_0)\\l(Q)\geq\gamma^*l(\mathcal{M}(Q_0))}}|Q| \\
        &\leq(1-\sigma^*)|Q_0| \\
        &<s|Q_0|.
    \end{align*}
    Thus, $\mathcal{S}_{Q_0}(s)>\gamma^* l(\mathcal{M}(Q_0))-\varepsilon$, for every $\varepsilon>0$, meaning that $\mathcal{S}_{Q_0}(s)\geq\gamma^* l(\mathcal{M}(Q_0))$. But then
    $$\mathcal{L}_{Q_0}(s)\leq l(\mathcal{M}(Q_0))\leq\frac{1}{\gamma^*}\mathcal{S}_{Q_0}(s),$$
    so the theorem follows by taking $C_0:=(\gamma^*)^{-1}$.
\end{proof}

\subsection{A probabilistic approach describing pore distribution}
\label{subsec6.2}

In this subsection we consider the distance weights associated with two subsets of the real line, introduced in the works of Anderson et al. \cite{ANDERSON} and Mudarra \cite{MUDARRA}, with the goal of presenting a probabilistic approach that may be useful in order to translate theory into practice. Both sets can be obtained through a same process involving the performance of a sequence of translations, contractions and reflections on the starting two-point set $\{0,1\}$ as is described next.

Given a sequence of ``contractions'' $\{c_n\}_{n\in\mathbb{N}}\subset(0,1]$, let
$$E(\{c_n\})=E^-(\{c_n\})\cup E^+(\{c_n\})$$
where $E^-(\{c_n\})=-E^+(\{c_n\})$ is the reflection of $E^+(\{c_n\})=\bigcup_{n\in\mathbb{N}_0}E_n$ through the origin and the sets $E_n$ used to construct the latter are defined as follows. Set $E_0:=\{0,1\}$ and, for each $n\geq1$, $E_n:=E_{n-1}\cup E^1_{n-1}\cup E^2_{n-1}$ where $E^1_{n-1}$ is a translation of $E_{n-1}$ contracted by $c_n$ and whose first point is the last point of $E_{n-1}$ and $E^2_{n-1}$ is a translation of $E_{n-1}$ whose first point is the last point of $E^1_{n-1}$.

\begin{figure}[b]
	\centering
	\includegraphics[width=0.9\textwidth]{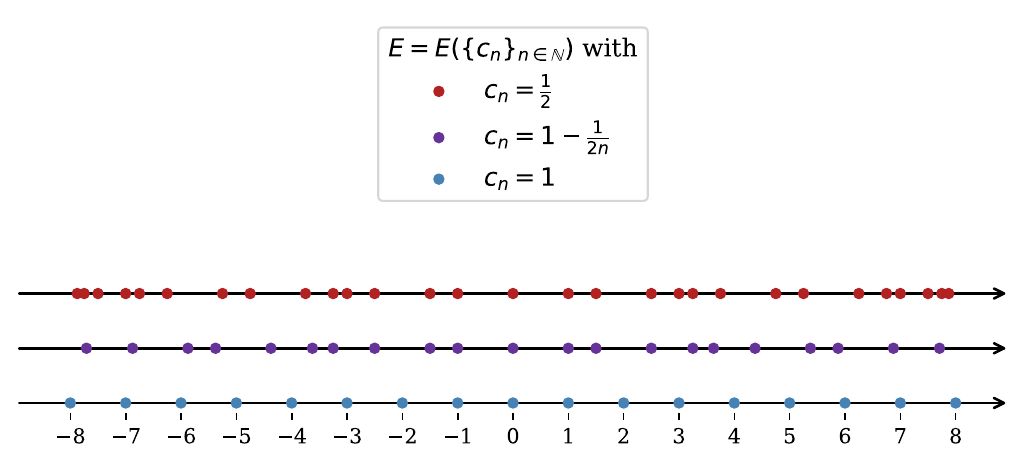}
	\caption{Points of the set $E = E(\{c_n\})$ contained in the interval $[-8, 8]$, corresponding to the sequences (from bottom to top) $c_n = 1$, $c_n = 1 - \frac{1}{2n}$, and $c_n=\frac{1}{2}$.}
	\label{fig1}
\end{figure}

Clearly, the nature of $E(\{c_n\})$ heavily depend on the behavior of the sequence of contractions considered in its construction. In particular, the constant sequence $c_n=1$ gives rise to the set $E(\{1\})=\mathbb{Z}$, which happens to be a weakly porous set \cite{ANDERSON}. This example is represented in Figure \ref{fig1} along with two others, constructed using the sequences $c_n=1-\frac{1}{2n}$ and $c_n=\frac{1}{2}$ instead. The set $E(\{1-\frac{1}{2n}\})$ was shown to satisfy $\dist(\cdot,E(\{1-\frac{1}{2n}\}))^{-\alpha}\in A_p\setminus A_1$ for every $0<\alpha<1$ and every $1<p<\infty$ in \cite{ANDERSON}, while $\dist(\cdot,E(\{\frac{1}{2}\}))^{-\alpha}$ was shown to be a doubling weight which is not contained in $\bigcup_{p\geq1}A_p$ for any choice of $0<\alpha<1$ in \cite{MUDARRA}. This shows that the sets $E(\{c_n\})$ are particularly useful for constructing examples of $A_1$, $A_p$ or merely doubling distance weights depending on the choice of the sequence $\{c_n\}$, resulting on more regular weights as the sequence approaches one more rapidly.

Here, we will restrict ourselves to the cases $c_n=1-\frac{1}{2}$ and $c_n=\frac{1}{2}$ to demonstrate how the theory developed in previous sections agree with the expected properties of the sets $E(\{c_n\})$. More precisely, if for each $n\in\mathbb{N}$, $Q_n$ denote the smallest cube (interval) in $\mathbb{R}$ containing the set $E_n$ (of course, the definition of each $Q_n$ depend upon the sequence $\{c_n\}$), the next statement can be proved.

\begin{proposition}
    \label{prop6}
    Let $E=E(\{c_n\}_{n\in\mathbb{N}})$ where $\{c_n\}_{n\in\mathbb{N}}$ is some sequence of contractions $0<c_n<1$.
    \begin{itemize}
        \item[(i)] If the sequence of contractions is given by $c_n=1-\frac{1}{2n}$, then there exist $C_0>0$ and $s\in(0,\frac{1}{1+2^n})$ such that
        $$\mathcal{L}_{Q_n}(s)\leq C_0\:\mathcal{S}_{Q_n}(s)\text{ for every }n\in\mathbb{N}_0.$$
        \item[(ii)] If the sequence of contractions is given by $c_n=\frac{1}{2}$, then for every $t\in(0,\frac{1}{8})$ we have
        $$\frac{\mathcal{L}_{Q_n}(t)}{\mathcal{S}_{Q_n}(t)}\to\infty\text{ as }n\to\infty.$$
    \end{itemize}
\end{proposition}

Item (ii) above implies that $E(\{\frac{1}{2}\})$ is not an $A_p$ set, since \eqref{eq7} cannot hold for any choice of $\sigma,C_0>0$, which agrees with the previous discussion. On the other hand, recall that the set considered in (i) does admit distance weights in $A_p$ for every $1<p<\infty$, but not in $A_1$. In particular, this set is not weakly porous but it should satisfy \eqref{eq10} for every cube $Q_0$ intersecting it. Naturally, (i) is a weaker statement than this condition, and extending it to every cube intersecting the set $E(\{1-\frac{1}{2n}\})$ would require a great deal of additional and technical results. Instead, the goal here is to merely prove \eqref{eq10} for the sequence of cubes $Q_n$ that, in some sense, constitute the structural core of the set in question. The tools to be used in the proof of Proposition \ref{prop6} are interesting in their own right, as they bridge the theory from previous sections with elementary probability concepts that could be suitable for a more practical viewpoint.

Let $E$ be a closed non-empty set in the real line and let $Q_0\subset\mathbb{R}$ be a cube. The collection
$$\mathcal{C}_E(Q_0)=\{I:I\text{ is a connected component of }Q_0^\circ\setminus E\}$$
consists of all the open intervals $I$ that happen to be the connected components of $Q_0^\circ\setminus E$, in contrast with $\mathcal{D}_E(Q_0)$ which consisted of dyadic intervals. Here, we use the notation $A^\circ$ to denote the interior of a set $A\subset\mathbb{R}$. Using this new collection, we can define analogs to the functions $\mathcal{L}_{Q_0}$ and $\mathcal{S}_{Q_0}$: for any $0<t<1$, let
$$\Tilde{\mathscr{L}}(t,Q_0,E)=\Bigg\{L\geq0:\sum_{\substack{I\in\mathcal{C}_E(Q_0)\\|I|\geq L}}\frac{|I|}{|Q_0|}\geq t\Bigg\}\hspace{0.25cm}$$
and
$$\Tilde{\mathscr{S}}(t,Q_0,E)=\Bigg\{L\geq0:\sum_{\substack{I\in\mathcal{C}_E(Q_0)\\|I|\leq L}}\frac{|I|}{|Q_0|}\geq t\Bigg\}.$$
Then, we consider the functions
\begin{equation*}
    \Tilde{\mathcal{L}}_{Q_0}(t):=\sup\Tilde{\mathscr{L}}(t,Q_0,E)
\end{equation*}
and
\begin{equation*}
    \Tilde{\mathcal{S}}_{Q_0}(t):=\inf\Tilde{\mathscr{S}}(t,Q_0,E).
\end{equation*}

The next result shows that these newly introduced quantities are naturally related to their dyadic counterparts that had been used so far. 

\begin{lemma}
    \label{lemma6.5}
    Let $E$ be a closed non-empty set with zero measure and let $0<t'<t<1$. Then, for every cube $Q_0$ intersecting $E$, the inequalities
    \begin{equation}
        \label{eq12}
        \mathcal{L}_{Q_0}(t)\leq\Tilde{\mathcal{L}}_{Q_0}(t)\leq4\mathcal{L}_{Q_0}(\tfrac{1}{4}t)
    \end{equation}
    and
    \begin{equation}
        \label{eq13}
        C(t,t')\:\Tilde{\mathcal{S}}_{Q_0}(t')\leq\mathcal{S}_{Q_0}(t)\leq\Tilde{\mathcal{S}}_{Q_0}(t)
    \end{equation}
    hold. The constant $C$ appearing in \eqref{eq13} depends only on $t$ and $t'$.
\end{lemma}

\begin{proof}
    We start by considering \eqref{eq12}. For the first inequality, take $L\in\mathscr{L}(t,Q_0,E)$. Then, since
    $$\bigcup_{\substack{Q\in\mathcal{D}_E(Q_0)\\l(Q)\geq L}}Q^\circ\subset\bigcup_{\substack{I\in\mathcal{C}_E(Q_0)\\|I|\geq L}}I,$$
    it follows that
    $$\sum_{|I|\geq L}\frac{|I|}{|Q_0|}\geq\sum_{l(Q)\geq L}\frac{|Q|}{|Q_0|}\geq t.$$
    In other words, $L\in\Tilde{\mathscr{L}}(t,Q_0,E)$ and, in consequence, $L\leq\Tilde{\mathcal{L}}_{Q_0}(t)$. Thus, taking the appropriate limit, we get $\mathcal{L}_{Q_0}(t)\leq\Tilde{\mathcal{L}}_{Q_0}(t)$. For the remaining inequality, take $L\in\Tilde{\mathscr{L}}(t,Q_0,E)$ and write $\{I_\alpha\}_{\alpha\in A}=\{I\in\mathcal{C}_E(Q_0):|I|\geq L\}$ for some adequate index set $A$. Since each $I_\alpha$ is an open set contained in $Q_0^\circ$, we can find a dyadic cube $Q_\alpha\in\mathcal{D}_E(Q_0)$ such that $Q_\alpha\subset I_\alpha$ and $|Q_\alpha|\geq\frac{1}{4}|I_\alpha|\geq\frac{1}{4}L$. Then,
    \begin{align*}
        \sum_{l(Q)\geq \frac{1}{4}L}\frac{|Q|}{|Q_0|}\geq\sum_{\alpha\in A}\frac{|Q_\alpha|}{|Q_0|}\geq\frac{1}{4}\sum_{\alpha\in A}\frac{|I_\alpha|}{|Q_0|}=\frac{1}{4}\sum_{|I|\geq L}\frac{|I|}{|Q_0|}\geq\frac{1}{4}t,
    \end{align*}
    which implies that $\frac{1}{4}L\in\mathscr{L}(\tfrac{1}{4}t,Q_0,E)$, and furthermore that $\Tilde{\mathcal{L}}_{Q_0}(t)\leq4\mathcal{L}_{Q_0}(\frac{1}{4}t)$.
    
    We now consider \eqref{eq13}. To prove the second inequality, let $L\in\Tilde{\mathscr{S}}(t,E,Q_0)$. Since
    $$\bigcup_{\substack{I\in\mathcal{C}_E(Q_0)\\|I|\leq L}}I\subset\bigcup_{\substack{Q\in\mathcal{D}_E(Q_0)\\l(Q)\leq L}}Q,$$
    we have
    $$\sum_{l(Q)\leq L}\frac{|Q|}{|Q_0|}\geq\sum_{|I|\leq L}\frac{|I|}{|Q_0|}\geq t.$$
    In other words, $L\in\mathscr{S}(t,Q_0,E)$ and, thus, $\mathcal{S}_{Q_0}(t)\leq L$. From this, it is deduced that $\mathcal{S}_{Q_0}(t)\leq\Tilde{\mathcal{S}}_{Q_0}(t)$. In order to prove the last inequality, let $0<L<\mathcal{S}_{Q_0}(t')$ and set $\mathcal{A}:=\{I\in\mathcal{C}_E(Q_0):|I|>L\}$. By hypothesis, 
    $$\sum_{I\in\mathcal{C}_E(Q_0)\setminus\mathcal{A}}\frac{|I|}{|Q_0|}<t'.$$
    On the other hand, we can write $\mathcal{D}_E(Q_0)=\bigcup_{I\in\mathcal{C}_E(Q_0)}\{Q:Q\subset I\}$, since for every $Q\in\mathcal{D}_E(Q_0)$ there exists a unique $I\in\mathcal{C}_E(Q_0)$ such that $Q\subset I$. Furthermore, if $I\in\mathcal{A}$ and $\#$ denotes the counting measure, then we claim that
    $$\#\Big(\Big\{Q\in\mathcal{D}_E(Q_0):Q\subset I\land\frac{1}{2^{k+1}}L<l(Q)\leq\frac{1}{2^k}L\Big\}\Big)\leq2$$
    for any $k\in\mathbb{N}_0$. To see this, assume $Q_i$ is a member of the former collection for $i\in\{1,2,3\}$ and write $I=(a,b)$. It follows that, since $\pi Q_i\cap E\neq\emptyset$ and $I\cap E=\emptyset$, at least two of these cubes intersect one same endpoint of $I$. Let us assume, without loss of generality, that $\pi Q_1$ and $\pi Q_2$ both contain $a$, implying $\pi Q_1\cap\pi Q_2\neq\emptyset$. But then, since $l(Q_1)=l(Q_2)$, necessarily $\pi Q_1=\pi Q_2$. Because $Q_1$ and $Q_2$ do not intersect $E$, the latter also implies $Q_1=Q_2$ and the claim holds. Thus, for $I\in\mathcal{A}$ and $k\in\mathbb{N}_0$, we have that
    \begin{align*}
        \Bigg|\bigcup_{\substack{Q\subset I\\l(Q)\leq\frac{1}{2^k}L }}Q\Bigg|=\sum_{\substack{Q\subset I\\l(Q)\leq\frac{1}{2^k}L }}|Q|=\sum_{j=k}^\infty\sum_{\substack{Q\subset I\\ \frac{1}{2^{j+1}}L<l(Q)\leq\frac{1}{2^j}L }}|Q|\leq\sum_{j=k}^\infty\frac{1}{2^{j-1}}L=\frac{1}{2^{k-2}}L<\frac{1}{2^{k-2}}|I|.
    \end{align*}
    Picking $k\in\mathbb{N}_0$ such that $t'+\frac{1}{2^{k-2}}<t$, we get
    \begin{align*}
        \sum_{\substack{Q\in\mathcal{D}_E(Q_0)\\l(Q)\leq\frac{1}{2^k}L}}|Q|= \sum_{I\in\mathcal{C}_E(Q_0)}\sum_{\substack{Q:Q\subset I\\l(Q)\leq\frac{1}{2^k}L}}|Q|=\sum_{I\in\mathcal{C}_E(Q_0)\setminus\mathcal{A}}\sum_{\substack{Q:Q\subset I\\l(Q)\leq\frac{1}{2^k}L}}|Q|+\sum_{I\in\mathcal{A}}\sum_{\substack{Q:Q\subset I\\l(Q)\leq\frac{1}{2^k}L}}|Q|.
    \end{align*}
    For the first series, we write
    \begin{align*}
        \sum_{I\in\mathcal{C}_E(Q_0)\setminus\mathcal{A}}\sum_{\substack{Q:Q\subset I\\l(Q)\leq\frac{1}{2^k}L}}|Q|\leq\sum_{I\in\mathcal{C}_E(Q_0)\setminus\mathcal{A}}|I|<t'|Q_0|.
    \end{align*}
    For the remaining one, we compute
    \begin{align*}
        \sum_{I\in\mathcal{A}}\sum_{\substack{Q:Q\subset I\\l(Q)\leq\frac{1}{2^k}L}}|Q|<\frac{1}{2^{k-2}}\sum_{I\in\mathcal{A}}|I|\leq\frac{1}{2^{k-2}}|Q_0|.
    \end{align*}
    Consequently,
    $$\sum_{\substack{Q\in\mathcal{D}_E(Q_0)\\l(Q)\leq\frac{1}{2^k}L}}|Q|\leq(t'+\tfrac{1}{2^{k-2}})|Q_0|<t|Q_0|,$$
    which implies that $\frac{1}{2^k}L<\mathcal{S}_{Q_0}(t)$. Taking the limit as $L\to\Tilde{\mathcal{S}}_{Q_0}(t'),$ we arrive at $\frac{1}{2^k}\Tilde{\mathcal{S}}_{Q_0}(t')\leq\mathcal{S}_{Q_0}(t)$.
\end{proof}

In view of Lemma \ref{lemma5}, it will be enough to characterize the behavior of $\Tilde{\mathcal{L}}_{Q_0}$ and $\Tilde{\mathcal{S}}_{Q_0}$ to deduce inequalities like the ones appearing in Proposition \ref{prop6}. With this in mind, given a cube $Q_0\subset\mathbb{R}$, let us consider the function $X_Q:Q\to(0,\infty)$ given by    
$$X_Q(\omega):=\sum_{I\in\mathcal{C}_E(Q)}|I|\mathbbm{1}_I(\omega).$$
Namely, $X_Q$ assigns, to each $\omega$ in $Q$, the length of the connected component of $Q^\circ\setminus E$ containing this point (or zero if this point is not contained in any such component). Furthermore, $X_Q$ happens to be a random variable on the probability space $(Q,\mathcal{B}(Q),P_{Q})$, where $\mathcal{B}(Q)$ is the Borel $\sigma$-algebra of subsets of $Q$ and $P_Q(S)=|Q|^{-1}|S|$ for any $S\in\mathcal{B}(Q)$. We will use the standard notation $\mathbb{E}X_Q$ for the expected value of $X_Q$ (with respect to the measure $P_Q$). We will now begin tackling the first statement of Proposition \ref{prop6}.

\begin{lemma}
    \label{lemma6}
    Let $E=E(\{1-\frac{1}{2n}\})$. There exists a constant $C>0$ such that
    \begin{equation}
        \label{eq11}
        \mathbb{E}X_{Q_n}\:\mathbb{E}X_{Q_n}^{-1}\leq C
    \end{equation}
    for every $n\in\mathbb{N}_0$.
\end{lemma}

\begin{proof}
    Let $x_1,\hdots,x_m>0$ be such that $P_{Q_n}(X_{Q_n}=x_i)=p_i$ with $\sum_{i=1}^mp_i=1$. Then, $X_{Q_{n+1}}(Q_{n+1})=\{x_1,\hdots,x_m,c_{n+1}x_1,\hdots,c_{n+1}x_m\}$ with $P_{Q_n}(X_{Q_{n+1}}=x_i)=(\frac{2}{2+c_{n+1}})p_i$ and $P_{Q_n}(X_{n+1}=c_{n+1}x_i)=(\frac{c_{n+1}}{2+c_{n+1}})p_i$. Thus, for any $\theta\in\mathbb{R}$, we have
    \begin{align*}
        \mathbb{E}X_{Q_{n+1}}^\theta=\sum_{y\in X_{Q_{n+1}}(Q_{n+1})}y^\theta P_{Q_{n+1}}(X_{Q_{n+1}}=y)&= \sum_{i=1}^m\Big(\tfrac{2}{2+c_{n+1}}x_i^\theta p_i+\tfrac{c_{n+1}}{2+c_{n+1}}c_{n+1}^\theta x_i^\theta p_i\Big) \\
        &= \Big(\frac{2+c_{n+1}^{1+\theta}}{2+c_{n+1}}\Big)\sum_{i=1}^mx_i^\theta p_i \\
        &=\Big(\frac{2+c_{n+1}^{1+\theta}}{2+c_{n+1}}\Big)\mathbb{E}X_{Q_n}^\theta.
    \end{align*}
    Using this recursive formula with $\theta=\pm1$, we can write
    \begin{align*}
        \mathbb{E}X_{Q_n}\:\mathbb{E}X_{Q_n}^{-1}&=\Big(\frac{2+c_n^2}{2+c_n}\Big)\Big(\frac{3}{2+c_n}\Big)\mathbb{E}X_{Q_{n-1}}\:\mathbb{E}X_{Q_{n-1}}^{-1} \\
        &\vdots \\
        &=\mathbb{E}X_{Q_0}\:\mathbb{E}X_{Q_0}^{-1}\prod_{i=1}^{n}\frac{3(2+c_i^2)}{(2+c_i)^2} \\
        &=\prod_{i=1}^{n}\frac{3(2+c_i^2)}{(2+c_i)^2},
    \end{align*}
    where the last step follows from the fact that $X_{Q_0}=1$. Replacing $c_i=1-\frac{1}{2i}$ in the previous formula and recalling that $\log(x+1)\leq x$ for every positive $x$, we find that
    \begin{align*}
        \mathbb{E}X_{Q_n}\:\mathbb{E}X_{Q_n}^{-1}=\prod_{i=1}^{n}\frac{3(2+c_i^2)}{(2+c_i)^2}=\prod_{i=1}^{n}\frac{3(2+(1-\frac{1}{2i})^2)}{(3+\frac{1}{2i})^2}&=\prod_{i=1}^{n}\frac{36i^2-12i+3}{36i^2-12i+1} \\
        &=\prod_{i=1}^{n}\Big(1+\frac{2}{36i^2-12i+1}\Big) \\
        &=e^{\sum_{i=1}^n\log\Big(1+\frac{2}{36i^2-12i+1}\Big)} \\
        &\leq e^{\sum_{i=1}^n\frac{2}{36i^2-12i+1}} \\
        &\leq e^{\sum_{i=1}^\infty\frac{2}{(6i-1)^2}} \\
        &=e^{\frac{2}{25}\sum_{i=1}^\infty\frac{1}{i^2}}.
    \end{align*}
    Thus, the proof finishes by taking the bound $C=e^{\frac{\pi^2}{75}}$.
\end{proof}

\begin{lemma}
    \label{lemma7}
    Let $E=E(\{1-\frac{1}{2n}\})$ and fix $0<s<(1+2^n)^{-1}$ . Then, there exists $C_0>0$ depending only on $s$ such that
    $$\Tilde{\mathcal{L}}_{Q_n}(s)\leq C_0\:\Tilde{\mathcal{S}}_{Q_n}(s)$$
    for every $n\in\mathbb{N}_0$.
\end{lemma}

\begin{proof}
    Fix and $0<s<(1+2^n)^{-1}$ let $k$ be any positive number. Observe that
    \begin{align*}
        P_{Q_n}(X_{Q_n}\geq k\:\mathbb{E}X_{Q_n})=P_{Q_n}\Big(\bigcup_{|I|\geq k\:\mathbb{E}X_{Q_n}}I\Big)=\sum_{|I|\geq k\:\mathbb{E}X_{Q_n}}\frac{|I|}{|Q_n|}.
    \end{align*}
    But then, if $k>s^{-1}$, we can use Markov's inequality to get
    \begin{align*}
        \sum_{|I|\geq k\:\mathbb{E}X_{Q_n}}\frac{|I|}{|Q_n|}=P_{Q_n}(X_{Q_n}\geq k\:\mathbb{E}X_{Q_n})\leq\frac{1}{k}<s,
    \end{align*}
    meaning that $\Tilde{\mathcal{L}}_{Q_n}(s)\leq k\:\mathbb{E}X_{Q_n}$. Similarly, since
    \begin{align*}
        P_{Q_n}\Big(X_{Q_n}^{-1}\geq\frac{\mathbb{E}X_{Q_n}^{-1}}{k'}\Big)=P_{Q_n}\Big(X_{Q_n}\leq\frac{k'}{\mathbb{E}X_{Q_n}^{-1}}\Big)=\sum_{|I|\leq k'\:(\mathbb{E}X_{Q_n}^{-1})^{-1}}\frac{|I|}{|Q_n|}
    \end{align*}
    we can apply Markov's inequality to the variable $X_{Q_n}^{-1}$ to obtain
    \begin{align*}
        \sum_{|I|\leq k'\:(\mathbb{E}X_{Q_n}^{-1})^{-1}}\frac{|I|}{|Q_n|}=P_{Q_n}\Big(X_{Q_n}^{-1}\geq\frac{\mathbb{E}X_{Q_n}^{-1}}{k'}\Big)\leq k'<s
    \end{align*}
    if we choose $0<k'<s$, thus $\Tilde{\mathcal{S}}_{Q_n}(s)\geq k'\:(\mathbb{E}X_{Q_n}^{-1})^{-1}$. But then, if $C$ is the constant given by Lemma \ref{lemma7},
    $$\Tilde{\mathcal{L}}_{Q_n}(s)\leq k\:\mathbb{E}X_{Q_n}\leq kC(\mathbb{E}X_{Q_n}^{-1})^{-1}\leq\frac{k}{k'}C\Tilde{\mathcal{S}}_{Q_n}(s).$$
\end{proof}

\begin{proof}[Proof of Proposition \ref{prop6} (i)]
    Follows by the application of Lemmas \ref{lemma6.5} and \ref{lemma7}.
\end{proof}

To tackle the second item of Proposition \ref{prop6} it is better to consider the random variable given by
$$Y_Q(\omega)=-\log_2 X_Q(\omega)$$
which, in the case that $E=E(\{\frac{1}{2}\})$, to each $\omega\in Q$ assigns an integer $k$ such that the connected component $I$ containing $\omega$ has length $2^{-k}$. It was already noted by Mudarra in \cite{MUDARRA} that the pore distribution of $E$ could be described by a binomial distribution. In our setting, this translates into the fact that $Y_{Q_n}$ follows a binomial distribution with parameters $n$ and $q=\frac{1}{5}$, this is, $Y_{Q_n}\sim\text{Bin}(n,q)$. In this regard, it is known that the binomial distribution can be uniformly
approximated by the normal distribution, and that this approximation improves
as the parameter $n$ increases. The main tool for achieving this approximation is the Berry-Esseen theorem, reproduced below.

\begin{theorem}[\textup{\cite[Theorem 11.4.1]{ATHREYA}}]
    \label{teobe}
    Let $X_1,X_2,\hdots$ be a sequence of iid random variables with $\mathbb{E}X_1=\mu$, $\text{Var}X_1=\sigma^2\in(0,\infty)$ and $\mathbb{E}|X_1|^3<\infty$. Then,
    $$\sup_{x\in\mathbb{R}}\Big| P\Big(\frac{X_1+\hdots+X_n-n\mu}{\sigma\sqrt{n}}\leq x\Big)-\Phi(x) \Big|\leq C\frac{\mathbb{E}|X_1-\mu|^3}{\sigma^3\sqrt{n}}$$
    for all $n\in\mathbb{N}$ and where $\Phi$ denotes the standard normal cumulative distribution function. Furthermore, the constant $C\in(0,\infty)$ does not depend on $n$ or the distribution of $X_1$.
\end{theorem}

Since any binomial random variable can be expressed as a sum of iid Bernoulli  random variables, Theorem \ref{teobe}, when applied to such a sequence of functions, ensures the existence of a constant $C>0$, independent of $n$, such that
\begin{equation}
    \label{eq14}
    \sup_{x\in\mathbb{R}}\Big|P_{Q_n}(Y_{Q_n}\leq x) - \Phi\Big(\frac{x-nq}{\sqrt{nq(1-q)}}\Big)\Big|\leq\frac{C}{\sqrt{n}}.
\end{equation}
This can be used to get approximate values for the functions $\Tilde{\mathcal{L}}$ and $\Tilde{\mathcal{S}}$ in the following way.

\begin{proof}[Proof of Proposition \ref{prop6} (ii)]
    Fix $0<t<\frac{1}{8}$ and begin by rewriting the expressions for the functions $\Tilde{\mathcal{L}}$ and $\Tilde{\mathcal{S}}$ in terms of the distribution of $Y_{Q_n}$. Given $L\geq0$, recall that $P_{Q_n}(X_{Q_n}\geq L)=P_{Q_n}(\{\omega\in Q_n:|I|\geq L,\text{ where }\omega\in I\in\mathcal{C}_E(Q_n)\})=\sum_{|I|\geq L}P_{Q_n}(I)=\sum_{I\geq L}\frac{|I|}{|Q_n|}$, thus
    \begin{align*}
        \Tilde{\mathcal{L}}_{Q_n}(t)=\sup\Big\{L\geq0:\sum_{\substack{I\in\mathcal{C}_E(Q_0)\\|I|\geq L}}\frac{|I|}{|Q_n|}\geq t\Big\} &=\sup\{L\in(0,1):P_{Q_n}(X_{Q_n}\geq L)\geq t\} \\
        &=\max\{\tfrac{1}{2^k}:k\in\mathbb{N}_0\land P_{Q_n}(X_{Q_n}\geq \tfrac{1}{2^k})\geq t\} \\
        &=2^{-\min\{k\in\mathbb{N}_0: P_{Q_n}(X_{Q_n}\geq \tfrac{1}{2^k})\geq t\}} \\
        &=2^{-\min\{k\in\mathbb{N}_0: P_{Q_n}(Y_{Q_n}\leq k)\geq t\}}.
    \end{align*}
    Analogously, it can be shown that $\Tilde{\mathcal{S}}_{Q_n}(t)=2^{-\max\{k\in\mathbb{N}_0: P_{Q_n}(Y_{Q_n}\geq k)\geq t\}}$. Then, by \eqref{eq12} and \eqref{eq13}, we have
    \begin{equation}
        \label{eq15}
        \frac{\mathcal{L}_{Q_n}(t)}{\mathcal{S}_{Q_n}(t)}\geq\frac{1}{4}\frac{\Tilde{\mathcal{L}}_{Q_n}(4t)}{\Tilde{\mathcal{S}}_{Q_n}(t)}=\frac{1}{4}2^{\max\{k\in\mathbb{N}_0: P_{Q_n}(Y_{Q_n}\geq k)\geq t\}-\min\{k\in\mathbb{N}_0:P_{Q_n}(Y_{Q_n}\leq k)\geq4t\}}.
    \end{equation}

    In order to get approximate values to the quantities in the exponent above, we consider $z_r$, the standard normal quantile function at $r$. In other words, this function is defined to satisfy $\Phi(z_r)=r$ for every $0<r<1$. Next, pick numbers $\xi_1$ and $\xi_2$ such that $t<\xi_1<4t<\xi_2<\tfrac{1}{2}$. If we set $x_{1-\xi_1}:=nq+(1-z_{\xi_1})\sqrt{nq(1-q)}$ and $x_{\xi_2}:=nq+z_{\xi_2}\sqrt{nq(1-q)}$, then by \eqref{eq14} it follows that
    $$|P_{Q_n}(Y_{Q_n}\leq x_{\xi_2})-\xi_2|\leq C\sqrt{n}^{-1}.$$
    In particular choosing $n>C^2(\xi_2-4t)^{-2}$, we have $P_{Q_n}(Y_{Q_n}\leq\xi_2)\geq4t$ and
    $$\min\{k\in\mathbb{N}_0: P_{Q_n}(Y_{Q_n}\leq k)\geq4t\}\leq x_{\xi_2}.$$
    On the other hand,
    $$|P_{Q_n}(Y_{Q_n}>x_{1-\xi_1})-\xi_1|=|P_{Q_n}(Y_{Q_n}\leq x_{1-\xi_1})-(1-\xi_1)|\leq C\sqrt{n}^{-1},$$
    so taking $n>C^2(\xi_1-t)^{-2}$ implies $P_{Q_n}(Y_{Q_n}>x_{1-\xi_1})\geq t$ and
    $$\max\{k\in\mathbb{N}_0: P_{Q_n}(Y_{Q_n}\geq k)\geq t\}\geq x_{1-\xi_1}.$$
    Finally, let us observe that $\xi_1<\xi_2<1-\xi_1$ and, consequently, $z_{1-\xi_1}-z_{\xi_2}>0$. Thus, combining \eqref{eq15} with the previous estimates we get
    $$\frac{\mathcal{L}_{Q_n}(t)}{\mathcal{S}_{Q_n}(s)}\geq\frac{1}{4}2^{x_{1-\xi_1}-x_{\xi_2}}=\frac{1}{4}2^{(z_{1-\xi_1}-z_{\xi_2})\sqrt{nq(1-q)}}\to\infty\text{ as }n\to\infty.$$
\end{proof}

\bibliographystyle{plain}
\bibliography{main}

\begin{thebibliography}{10}

\bibitem{AIKAWA}
H.~Aikawa.
\newblock Quasiadditivity of {R}iesz capacity.
\newblock {\em Mathematica Scandinavica}, 69(1):15--30, 1991.

\bibitem{AIMAR}
H.~Aimar, M.~Carena, R.~Dur\'{a}n, and M.~Toschi.
\newblock Powers of distances to lower dimensional sets as {M}uckenhoupt
  weights.
\newblock {\em Acta Math. Hungar.}, 143(1):119--137, 2014.

\bibitem{NOSOTROS}
H.~Aimar, I.~Gómez, and I.~Gómez~Vargas.
\newblock Weakly porous sets and {$A_1$} {M}uckenhoupt weights in spaces of
  homogeneous type, 2024.
\newblock arXiv 2406.14369.

\bibitem{NOS_LATERAL}
H.~Aimar, I.~Gómez, I.~{Gómez Vargas}, and F.~J. Martín-Reyes.
\newblock One-sided {M}uckenhoupt weights and one-sided weakly porous sets in
  $\mathbb{R}$.
\newblock {\em Journal of Functional Analysis}, 289(10):111110, 2025.

\bibitem{ANDERSON}
T.~C. Anderson, J.~Lehrb\"{a}ck, C.~Mudarra, and A.~V. V\"{a}h\"{a}kangas.
\newblock Weakly porous sets and {M}uckenhoupt {$A_p$} distance functions.
\newblock {\em J. Funct. Anal.}, 287(8):Paper No. 110558, 34, 2024.

\bibitem{ATHREYA}
K.~B. Athreya and S.~N. Lahiri.
\newblock {\em Measure theory and probability theory}.
\newblock Springer Texts in Statistics. Springer, New York, NY, 2006 edition,
  July 2006.

\bibitem{DUOANDIKOETXEA}
J.~Duoandikoetxea.
\newblock {\em Forty years of {M}uckenhoupt weights}.
\newblock {F}unction Spaces and Inequalities. (Paseky nad Jizerou, June 2013),
  Lecture Notes (J. Lukeš and L. Pick, eds.), Matfyzpress, Prague.

\bibitem{LOPEZ}
R.~G. Dur\'{a}n and F.~L\'{o}pez~Garc\'{\i}a.
\newblock Solutions of the divergence and analysis of the {S}tokes equations in
  planar {H}\"{o}lder-{$\alpha$} domains.
\newblock {\em Math. Models Methods Appl. Sci.}, 20(1):95--120, 2010.

\bibitem{DYDA}
B.~Dyda, L.~Ihnatsyeva, J.~Lehrb\"{a}ck, H.~Tuominen, and A.~V.
  V\"{a}h\"{a}kangas.
\newblock Muckenhoupt {$A_p$}-properties of distance functions and applications
  to {H}ardy-{S}obolev--type inequalities.
\newblock {\em Potential Anal.}, 50(1):83--105, 2019.

\bibitem{VAHAKANGAS}
J.~Kinnunen, J.~Lehrb\"ack, and A.~V\"ah\"akangas.
\newblock {\em Maximal Function Methods for Sobolev Spaces}, volume 257 of {\em
  Mathematical Surveys and Monographs}.
\newblock American Mathematical Society, Providence, RI, 2021.

\bibitem{MUDARRA}
C.~Mudarra.
\newblock Weak porosity on metric measure spaces, 2024.
\newblock arXiv 2306.11419.

\bibitem{PASQUARIELLO}
M.~Pasquariello and I.~Uriarte-Tuero.
\newblock Medians, oscillations, and distance functions, 2025.
\newblock arXiv 2507.21020.

\bibitem{VASIN}
A.~V. Vasin.
\newblock The limit set of a {F}uchsian group and {D}yn{'}kin{'s} lemma.
\newblock {\em Zap. Nauchn. Sem. S.-Peterburg. Otdel. Mat. Inst. Steklov.
  (POMI)}, 303(Issled. po Line\u{\i}n. Oper. i Teor. Funkts. 31):89--101, 322,
  2003.

\end{thebibliography}


%
%


\subsection*{Acknowledgements}
The author would like to thank Carlos Mudarra and Antti Vähäkangas for their insightful comments and suggestions on earlier drafts of this paper. This work was financially supported by Consejo Nacional de Investigaciones Cient\'ificas y T\'ecnicas-CONICET in Argentina. 


\bigskip

\bigskip

\noindent{\textit{Affiliation.} 
	\textsc{Instituto de Matem\'{a}tica Aplicada del Litoral ``Dra. Eleonor Harboure'', CONICET, UNL.}

	\noindent \textit{Address.} \textmd{IMAL, Streets F.~Leloir and A.P.~Calder\'on, CCT CONICET Santa Fe, Predio ``Alberto Cassano'', Colectora Ruta Nac.~168 km~0, Paraje El Pozo, S3007ABA Santa Fe, Argentina.}
	
	%
	\noindent \textit{E-mail:} \verb|ignaciogomez@santafe-conicet.gov.ar|
}

\end{document}